\documentclass[10pt]{article}

\usepackage[table]{xcolor}
\makeatletter
\let\insert@pcolumn\insert@column
\makeatother
\usepackage[margin=1in]{geometry}
\usepackage{amsmath,amssymb,amsthm,mathtools}
\usepackage{booktabs}
\usepackage{longtable}
\usepackage{makecell}
\usepackage{algorithm}
\usepackage{algpseudocode}
\usepackage{wrapfig}
\usepackage{microtype}
\usepackage{fontawesome5}
\usepackage{pifont}
\usepackage[hidelinks]{hyperref}
\usepackage[
  backend=biber,
  style=apa,
  natbib=true,
  maxcitenames=2,
  mincitenames=1,
  uniquename=false,
  sortcites=true
]{biblatex}
\DeclareFieldFormat{addendum}{#1\nopunct}
\renewbibmacro*{finentry}{\finentry\par}
\makeatletter
\@ifundefined{localrefcontext}{}{%
  \AtBeginRefsection{%
    \GenRefcontextData{sorting=ynt,sortingnamekeytemplatename=apasortcite}}%
  \AtEveryCite{%
    \localrefcontext[sorting=ynt,sortingnamekeytemplatename=apasortcite]}}
\makeatother
\usepackage{fancyhdr}
\usepackage[most]{tcolorbox}

\floatstyle{ruled}
\restylefloat{algorithm}

\newtheorem{theorem}{Theorem}

\newtheorem{lemma}{Lemma}
\newtheorem{corollary}[theorem]{Corollary}
\newtheorem{assumption}{Assumption}
\newtheorem{definition}{Definition}
\newtheorem{remark}{Remark}

\newcommand{\R}{\mathbb{R}}
\newcommand{\E}{\mathbb{E}}
\newcommand{\cF}{\mathcal{F}}
\newcommand{\cZ}{\mathcal{Z}}
\newcommand{\argmin}{\mathop{\rm arg\,min}}
\newcommand{\cmark}{\ding{51}}
\newcommand{\xmark}{\ding{55}}
\newcommand{\uniformtableformat}{%
  \renewcommand{\arraystretch}{1}%
  \setlength{\aboverulesep}{3pt}%
  \setlength{\belowrulesep}{3pt}%
  \setlength{\belowcaptionskip}{6pt}%
}
\colorlet{boxbordergreen}{green!45!black}
\hypersetup{
  colorlinks=true,
  linkcolor=boxbordergreen,
  citecolor=boxbordergreen,
  urlcolor=boxbordergreen,
  filecolor=boxbordergreen
}

\AtEveryCitekey{\color{boxbordergreen}}
\newrobustcmd*{\mkbibparensgreen}[1]{%
  \textcolor{boxbordergreen}{\mkbibparens{#1}}}
\DeclareCiteCommand{\parencite}[\mkbibparensgreen]
  {\usebibmacro{cite:init}%
   \usebibmacro{prenote}%
   \toggletrue{apa:inpcite}}
  {\usebibmacro{citeindex}%
   \usebibmacro{cite}%
   \usebibmacro{cite:post}}
  {}
  {\usebibmacro{postnote}%
   \togglefalse{apa:inpcite}}
\DeclareMultiCiteCommand{\parencites}[\mkbibparensgreen]{\parencite}
  {\setunit{\multicitedelim}}
\DeclareCiteCommand*{\parencite}[\mkbibparensgreen]
  {\usebibmacro{cite:init}%
   \usebibmacro{prenote}%
   \toggletrue{apa:inpcite}}
  {\usebibmacro{citeindex}%
   \usebibmacro{citeyear}%
   \usebibmacro{cite:post}%
   \togglefalse{apa:inpcite}}
  {}
  {\usebibmacro{postnote}}
\newtcolorbox{promptbox}{
  enhanced,
  breakable,
  colback=green!3!white,
  colframe=boxbordergreen,
  colbacktitle=boxbordergreen,
  coltext=black,
  coltitle=green!3!white,
  fonttitle=\bfseries,
  title={Homework problem context and documented prompt to ChatGPT~5.4 TE},
  boxrule=0.6pt,
  arc=1mm,
  left=1.5mm,
  right=1.5mm,
  top=1mm,
  bottom=1mm
}

\newtcolorbox{contextbox}{
  enhanced,
  breakable,
  colback=green!3!white,
  colframe=boxbordergreen,
  coltext=black,
  boxrule=0.6pt,
  arc=1mm,
  left=1.5mm,
  right=1.5mm,
  top=1mm,
  bottom=1mm
}

\newtcolorbox{keyinsightbox}{
  enhanced,
  colback=green!3!white,
  colframe=boxbordergreen,
  coltext=black,
  boxrule=0.6pt,
  arc=1mm,
  left=1.5mm,
  right=1.5mm,
  top=1mm,
  bottom=1mm,
  before skip=0.5\baselineskip,
  after skip=0.5\baselineskip
}
\title{\Large Improved Convergence Rates for Stochastic Multi-Gradient Descent \\which Close the Gap: A Proof by AI}
\author{\normalfont\normalsize
Lisha Chen\\
\normalfont\normalsize
University of Rochester,
Rochester, New York 14627\\
\normalfont\normalsize
lisha.chen@rochester.edu}
\date{}

\begin{document}
\maketitle

\begin{abstract}
For smooth nonconvex stochastic multi-objective optimization,
stochastic multi-gradient descent (SMG) computes an approximate steepest common descent direction of the objectives from stochastic gradients.  
Under standard assumptions, this note establishes 
a new convergence rate for SMG in terms of the 
squared Pareto-stationarity (PS) measure. With a constant stepsize 
and linearly growing mini-batches, the expected squared empirical PS measure at the algorithm's output is $\widetilde O(T^{-1})$ after $T$ iterations. This improves on the $\widetilde O(T^{-1/4})$ bound obtained by \citet{chen2024threeway} for linearly growing batches, without requiring bounded gradients. Here, $\widetilde O(\cdot)$ suppresses logarithmic factors.

\begin{keyinsightbox}
The key to the rate improvement is to 
exploit the Lipschitz continuity of the PS measure, defined by the norm of the multi-gradient descent algorithm (MGDA) direction, rather than the $(1/2)$-H{\"o}lder continuity of the MGDA direction used by~\citet{chen2024threeway}.
\end{keyinsightbox}

For variants of MoCo and MoCo+ with exact MGDA updates and constant batch sizes, the analysis gives expected squared empirical PS rates of $O(T^{-1/2})$ and $\widetilde O(T^{-2/3})$, respectively. These bounds match the corresponding single-objective momentum rates up to logarithmic factors.

The proof was accidentally discovered while the author was preparing homework for a graduate course: ChatGPT~5.4 Thinking Extended generated the initial proof strategy for the main SMG result in response to an author-written homework-solution prompt; the author then verified and reorganized the resulting argument. The appendices document the prompt, summarize the student submissions with LLM disclosures,
and present additional proofs and results on momentum-based stochastic MGDA.
\end{abstract}

\noindent{\bf Mathematics Subject Classification (2020):} 
90C29, 90C15, 90C26,  90C30.

\noindent{\bf Keywords:} multi-objective optimization, stochastic optimization,
multi-gradient descent algorithm (MGDA), Pareto stationarity, growing mini-batches, momentum, artificial intelligence (AI)-assisted proof.
\vspace{0.12in}

\section{Introduction}
\label{sec:introduction}

Multi-objective optimization (MOO) seeks solutions that balance several objectives~\citep{ehrgott2005}. Because objective gradients may conflict, a direction that decreases one objective can increase another. The multi-gradient descent algorithm (MGDA) addresses this issue by selecting the negative minimum-norm convex combination of the objective gradients~\citep{mukai1980,fliege2000steepest,desideri2012mgda}.
At a non-Pareto-stationary point, this direction decreases every objective
simultaneously and is also known as the steepest common descent direction
or the conflict-avoidant direction.
Other related multi-objective descent, projected-gradient, and Newton methods are
surveyed by~\citet{fukuda2014survey}.

Since full-batch objective gradients are often unavailable or too expensive in large-scale learning, \citet{mercier2018smgda} study a stochastic extension of MGDA for the expected objectives:
\begin{equation}
  \min_{x\in\R^p} F(x):=\big(f_1(x),\ldots,f_M(x)\big)^\top,
  \qquad
  f_m(x):=\E_{z}\big[f_{z,m}(x)\big],\quad m\in[M].
  \label{eq:mercier-stochastic-context}
\end{equation}
Here, $M$ is the number of objectives, $[M]:=\{1,\ldots,M\}$ is the objective index set, $x\in\R^p$ is the decision variable with dimension $p$, $z$ is a random data sample, $f_{z,m}(x)$ is the sample loss for objective $m$, and $\E_z$ denotes expectation with respect to the distribution of $z$.

Although the mini-batch gradient matrix is conditionally unbiased, the resulting MGDA direction is generally biased because the minimum-norm selection is nonlinear~\citep{mercier2018smgda,liu2024smg}.
The stochastic variants of MGDA  must control this bias to obtain convergence guarantees.
The vanilla stochastic multi-gradient descent (SMG) algorithm~\citep{liu2024smg}
addresses this by increasing the batch size at each iteration.
This note establishes a finite-time convergence-rate guarantee 
for the vanilla SMG algorithm,
improving on the prior rate of \citet{chen2024threeway} for linearly growing mini-batches without requiring bounded gradients.
ChatGPT~5.4 Thinking Extended (ChatGPT~5.4 TE) generated the initial proof strategy, which the author then verified and reorganized.
This note is organized as follows: Section~\ref{sec:problem-setting} presents the problem setting and notation, Section~\ref{sec:convergence} develops the convergence analysis, Section~\ref{sec:related-work} reviews related work, and Section~\ref{sec:conclusion} concludes. Appendix~\ref{app:stochastic-mgda-comparison} compares related work, Appendix~\ref{app:momentum-exact-mgda} extends the analysis to MoCo and MoCo+ with exact MGDA updates, Appendix~\ref{app:proof-discovery} documents the proof-discovery process, and Appendix~\ref{app:student-submissions} summarizes the student homework submissions.

\section{Problem setting and notation}
\label{sec:problem-setting}

Let $\cZ$ be the data space, and let $S=(z_1,\ldots,z_n)\in\cZ^n$ be a given training dataset. Recall that $[M]:=\{1,\ldots,M\}$, where $M\geq 2$. For each $m\in[M]$, let $f_{z,m}:\R^p\to\R$ denote the loss for datum $z$ and objective $m$, and define $F_z(x):=\big(f_{z,1}(x),\ldots,f_{z,M}(x)\big)^\top$. For a mini-batch $Z=(\zeta_1,\ldots,\zeta_{|Z|})$ sampled with replacement from $S$, define
\begin{equation}
 F_S(x):=\frac{1}{n}\sum_{i=1}^n F_{z_i}(x),
 \qquad
 F_Z(x):=\frac{1}{|Z|}\sum_{r=1}^{|Z|}F_{\zeta_r}(x).
 \label{eq:empirical-objective}
\end{equation}
For each $m\in[M]$, the scalar components of $F_S$ and $F_Z$ are
$f_{S,m}(x):=n^{-1}\sum_{i=1}^n f_{z_i,m}(x)$ and
$f_{Z,m}(x):=|Z|^{-1}\sum_{r=1}^{|Z|}f_{\zeta_r,m}(x)$, respectively.
The gradient matrices $\nabla F_S(x),\nabla F_Z(x)\in\R^{p\times M}$ have $m$-th columns $\nabla f_{S,m}(x)$ and $\nabla f_{Z,m}(x)$, respectively. Let
$\Delta^M:=\big\{\lambda\in\R^M:\lambda\geq 0,\ \mathbf{1}^\top\lambda=1\big\},$
where $\mathbf{1}$ is the all-ones vector. Throughout, $\langle\cdot,\cdot\rangle$ and $\|\cdot\|$ denote the Euclidean inner product and norm. For a matrix $A$, $\|A\|_2:=\max_{\|v\|=1}\|Av\|$ denotes its spectral norm and $\|A\|_F$ denotes its Frobenius norm.
Table~\ref{tab:notation} summarizes the notation used repeatedly below.
\vspace{-5mm}
\begin{table}[H]
\uniformtableformat
\setlength{\belowcaptionskip}{6pt}
\caption{Notation.}
\label{tab:notation}
\centering
\small
\begin{tabular}{m{0.24\linewidth}m{0.66\linewidth}}
\toprule
Notation & Description \\
\midrule
$x\in\R^p$ & Model parameter or decision variable. \\
$S, Z_t$ & Training dataset and mini-batch at iteration $t$. \\
$f_{S,m},f_{Z_t,m}$ & Empirical and mini-batch scalar objectives for objective $m$. \\
$F_S,F_{Z_t}$ & Empirical and mini-batch vector objectives. \\
$q_{t,m},Q_t$ & Mini-batch gradient $q_{t,m}:=\nabla f_{Z_t,m}(x_t)$ for objective $m$ and gradient matrix $Q_t:=\nabla F_{Z_t}(x_t)$. \\
$\lambda^\star_Q$ & A selected MGDA weight in $\argmin_{\lambda\in\Delta^M}\|Q\lambda\|^2$. \\
$d_Q$ & Conflict-avoidant direction $-Q\lambda^\star_Q$. \\
$r(Q)=\|d_Q\|$ & Pareto-stationarity (PS) measure associated with $Q$. \\
$R_S(x)$ & Squared empirical PS measure $r(\nabla F_S(x))^2$. \\
\bottomrule
\end{tabular}
\end{table}

In the standard componentwise order, $\bar x$ Pareto dominates $x$ if $F_S(\bar x)\leq F_S(x)$ and at least one component inequality is strict,
where $\leq$ denotes componentwise inequality. A Pareto-optimal point has no Pareto dominator~\citep{pareto1906manuale,ehrgott2005}. The corresponding first-order
stationarity notion, characterized by the absence of a common descent direction, is Pareto stationarity.

\begin{definition}[Empirical Pareto-stationarity measure]
A point $x\in\R^p$ is Pareto stationary for $F_S$ if $\nabla F_S(x)\lambda=0$ for some $\lambda\in\Delta^M$. Its squared Pareto-stationarity measure is
\begin{equation}
 R_S(x):=\min_{\lambda\in\Delta^M}\big\|\nabla F_S(x)\lambda\big\|^2.
 \label{eq:ps-measure}
\end{equation}
\end{definition}

Under differentiability, Pareto stationarity is a necessary first-order condition for local weak Pareto optimality~\citep{ehrgott2005}.
By definition, $x$ is Pareto stationary for $F_S$
if and only if the PS measure $ R_S(x) = 0$.

\begin{wrapfigure}[18]{r}{0.41\textwidth}
\vspace{-0.8cm}
\begin{minipage}{0.4\textwidth}
\begin{algorithm}[H]
\small
\caption{SMG: stochastic multi-gradient descent}
\label{alg:smg}
\begin{algorithmic}[1]
\State \textbf{Input:} training dataset $S$, initial model $x_0$, horizon $T$, stepsizes $\{\alpha_t\}_{t=0}^{T-1}$, and batch sizes $\{|Z_t|\}_{t=0}^{T-1}$.
\For{$t=0,\ldots,T-1$}
  \State Sample $Z_t$ from $S$ with replacement.
  \State Compute $q_{t,m}:=\nabla f_{Z_t,m}(x_t)$, $m\in[M]$, and form $Q_t$.
  \State Compute the MGDA weight $\lambda_t$ and direction $d_t$ by \eqref{eq:smg-weight} and \eqref{eq:smg-direction}, respectively.
  \State Update $x_{t+1}$ by \eqref{eq:smg-update}.
\EndFor
\State Sample $\tau\in\{0,\ldots,T-1\}$ independently with $\Pr(\tau=t)=\alpha_t/\sum_{s=0}^{T-1}\alpha_s$.
\State \textbf{Output:} $x_\tau$.
\end{algorithmic}
\end{algorithm}
\end{minipage}
\vspace{-0.8\baselineskip}
\end{wrapfigure}

\subsection{Stochastic multi-gradient descent}

Since SMG is a stochastic variant of MGDA,
we begin by introducing the basics of MGDA.
MGDA is a deterministic MOO algorithm with access to full-batch gradients. At iteration $t$, it evaluates
$\nabla f_{S,m}(x_t)$ for every objective $m\in[M]$, computes a direction that
minimizes the worst regularized directional derivative, and updates
$x_{t+1}=x_t+\alpha_t d_{\nabla F_S(x_t)}$. For a generic point
$x\in\R^p$, the MGDA direction is
\begin{equation}
 d_{\nabla F_S(x)}
 :=\argmin_{d\in\R^p}\max_{m\in[M]}
 \left\{\left\langle \nabla f_{S,m}(x),d\right\rangle
 +\frac{1}{2}\|d\|^2\right\}.
 \label{eq:empirical-mgda-direction}
\end{equation}

Rather than solving \eqref{eq:empirical-mgda-direction} directly, MGDA can
compute the same direction through a minimum-norm convex combination of the
objective gradients. Specifically, for a gradient matrix
$Q=[q_1,\ldots,q_M]\in\R^{p\times M}$, an associated MGDA weight and direction
satisfy~\citep{mukai1980,fliege2000steepest}
\begin{equation}
 \lambda^\star_Q\in\argmin_{\lambda\in\Delta^M}\|Q\lambda\|^2,
 \qquad
 d_Q:=-Q\lambda^\star_Q.
 \label{eq:matrix-mgda}
\end{equation}
For deterministic MGDA, $Q=\nabla F_S(x_t)$.
Vanilla SMG replaces the deterministic full-batch gradient matrix in
\eqref{eq:matrix-mgda} with unbiased stochastic gradients~\citep{liu2024smg}. At iteration $t$,
it samples a mini-batch $Z_t$ and computes
$q_{t,m}:=\nabla f_{Z_t,m}(x_t)$ for each $m\in[M]$, forming
$Q_t:=\nabla F_{Z_t}(x_t)=\big[q_{t,1},\ldots,q_{t,M}\big]$.
Throughout, the evaluation of $q_{t,m}$ at the current iterate $x_t$ is left implicit.
Under uniform sampling with replacement, $q_{t,m}$ is an unbiased
estimator of $\nabla f_{S,m}(x_t)$, as formalized in
Assumption~\ref{ass:oracle} below. Replacing $Q$ by $Q_t$ in
\eqref{eq:matrix-mgda} gives the following SMG update:
\begin{subequations}
\begin{align}
\lambda_t&\in\argmin_{\lambda\in\Delta^M}\|Q_t\lambda\|^2,
\label{eq:smg-weight}\\
d_t&:=-Q_t\lambda_t,
\label{eq:smg-direction}\\
x_{t+1}&:=x_t+\alpha_t d_t.
\label{eq:smg-update}
\end{align}
\end{subequations}
Algorithm~\ref{alg:smg} summarizes SMG, and \eqref{eq:smg-weight} is solved
exactly at each iteration. Although $\lambda_t$ may be nonunique,
$d_t=d_{Q_t}$ is uniquely defined and continuous in
$Q_t$~\citep{desideri2012mgda,svaiter2018holder,chen2024threeway}. For
deterministic $x_0$, let $\cF_t$ be the {$\sigma$-algebra} generated by
$x_0,Z_0,\ldots,Z_{t-1}$, so the iterates generated by
\eqref{eq:smg-weight}--\eqref{eq:smg-update} are adapted to
$\{\cF_t\}_{t\geq 0}$. After $T$ iterations, Algorithm~\ref{alg:smg} returns
$x_\tau$ from $x_0,\ldots,x_{T-1}$ with probabilities proportional to the
stepsizes. We impose the following standard
assumptions~\citep{chen2024threeway,liu2024smg}.

\begin{assumption}[Smoothness and lower bounds]
For every $z\in S$ and $m\in[M]$, $f_{z,m}$ is $L$-smooth. Each empirical objective is bounded below: $f_{S,m}^{\inf}:=\inf_{x\in\R^p}f_{S,m}(x)>-\infty$. Consequently, every empirical and mini-batch objective is also $L$-smooth.
\label{ass:smooth-lower}
\end{assumption}

\begin{assumption}[Stochastic gradients and mini-batch sampling]
\label{ass:oracle}
For every $t\geq 0$, the mini-batch size satisfies $|Z_t|\geq 1$.
Conditional on $\cF_t$, $Z_t$ is sampled uniformly with replacement from $S$
independently of the past. There exists a constant 
$\sigma\geq 0$ such that, almost
surely, for every $m\in[M]$,
\begin{align}
 &\E\big[q_{t,m}\mid\cF_t\big]=\nabla f_{S,m}(x_t),
 \label{eq:unbiasedness}\\
 &\E\big[\|q_{t,m}-\nabla f_{S,m}(x_t)\|^2\mid\cF_t\big]
 \leq \frac{\sigma^2}{|Z_t|}.
 \label{eq:variance}
\end{align}
\end{assumption}

Under uniform sampling with replacement, \eqref{eq:unbiasedness} holds
automatically. Moreover, \eqref{eq:variance} follows from the standard uniform
per-sample variance bound
$\sup_{x\in\R^p}\max_{m\in[M]}\E_{z\sim\operatorname{Unif}(S)}
[\|\nabla f_{z,m}(x)-\nabla f_{S,m}(x)\|^2]\leq\sigma^2$: independence of
the $|Z_t|$ sampled gradients yields the factor $1/|Z_t|$. Thus, the
mini-batch variance vanishes as $|Z_t|$ grows.

\subsection{Auxiliary lemmas}
This section presents the auxiliary lemmas used in the convergence proof.
The first property is the common-descent inequality for the mini-batch MGDA direction.

\begin{lemma}[Common descent direction inequality]
\label{lem:stochastic-common-descent}
For every $m\in[M]$, every iteration $t$, and every realization of the mini-batch,
\begin{equation}
 q_{t,m}^\top d_t\leq-\|d_t\|^2.
 \label{eq:common-descent}
\end{equation}
\end{lemma}

\begin{proof}
Let $p_t:=-d_t=Q_t\lambda_t$. Since $p_t$ is the projection of the origin onto $\operatorname{conv}\{q_{t,1},\ldots,q_{t,M}\}$, projection optimality gives
\begin{equation}
 \big\langle p_t,q_{t,m}-p_t\big\rangle\geq 0,
 \qquad m\in[M].
 \label{eq:projection-optimality}
\end{equation}
Rearranging \eqref{eq:projection-optimality} and using $d_t=-p_t$ yields \eqref{eq:common-descent}.
\end{proof}

The next result concerns the PS measure $r(Q)=\|d_Q\|$. The direction map
$Q\mapsto d_Q$ need not be Lipschitz but is generally only $(1/2)$-H{\"o}lder continuous~\citep{svaiter2018holder,chen2024threeway,huangchen2026more}. The PS measure has a stronger continuity guarantee stated below.

\begin{lemma}[Lipschitz continuity of the PS measure]
\label{lem:ps-measure-lipschitz}
For $Q,Q'\in\R^{p\times M}$, define $r(Q):=\min_{\lambda\in\Delta^M}\|Q\lambda\|
 =\|d_Q\|$. Then
\begin{equation}
 \big|r(Q)-r(Q')\big|\leq\|Q-Q'\|_2\leq\|Q-Q'\|_F.
 \label{eq:ps-measure-lipschitz}
\end{equation}
\end{lemma}

\begin{proof}
For any $\lambda\in\Delta^M$, the triangle inequality and
$\|\lambda\|_2\leq\|\lambda\|_1=1$ imply
\begin{equation}
 \|Q\lambda\|\leq\|Q'\lambda\|+\|(Q-Q')\lambda\|
 \leq\|Q'\lambda\|+\|Q-Q'\|_2.
 \label{eq:ps-measure-one-sided}
\end{equation}
The simplex $\Delta^M$ is compact and $\lambda\mapsto\|Q'\lambda\|$ is continuous, so there exists
$\lambda'\in\argmin_{\lambda\in\Delta^M}\|Q'\lambda\|$. Evaluating the minimum that defines $r(Q)$ at this \emph{fixed} comparator and then applying \eqref{eq:ps-measure-one-sided} with $\lambda=\lambda'$ gives
\begin{align}
r(Q) &\leq \|Q\lambda'\|
\leq \|Q'\lambda'\|+\|Q-Q'\|_2
= r(Q')+\|Q-Q'\|_2.
\label{eq:ps-measure-one-sided-min}
\end{align}
Repeating the same argument after interchanging $Q$ and $Q'$ yields
$r(Q')\leq r(Q)+\|Q-Q'\|_2$. Thus,
$|r(Q)-r(Q')|\leq\|Q-Q'\|_2$. The second inequality in
\eqref{eq:ps-measure-lipschitz} follows from $\|Q-Q'\|_2\leq\|Q-Q'\|_F$.
\end{proof}

Note that the Frobenius-norm bound is included in \eqref{eq:ps-measure-lipschitz}
because Assumption~\ref{ass:oracle} controls the conditional second moment of
the mini-batch gradient-matrix error in the Frobenius norm. This bound is used
directly in the convergence proof below. Lemma~\ref{lem:ps-measure-lipschitz}
differs from Theorem~3.1 of \citet{svaiter2018holder}: it establishes
Lipschitz continuity of the PS measure with respect to the gradient matrix
$Q$, whereas that result analyzes continuity of the MGDA direction with
respect to the model parameter $x$. 
However, the intermediate steps in the proof of Theorem~3.1 of \citet{svaiter2018holder}
already imply Lemma~\ref{lem:ps-measure-lipschitz}.
Although equivalent statements and proofs
of Lemma~\ref{lem:stochastic-common-descent} appear in prior
work~\citep{fliege2000steepest,svaiter2018holder,fernando2023moco,chen2024threeway}, we include it for completeness.

\section{Convergence to empirical Pareto stationarity}
\label{sec:convergence}

Under Assumptions~\ref{ass:smooth-lower}--\ref{ass:oracle}, the main result
controls the squared empirical PS measure for general mini-batch schedules.
Define the following quantities:
\begin{equation}
 A_T:=\sum_{t=0}^{T-1}\alpha_t,
 \qquad
 V_T:=\sum_{t=0}^{T-1}\frac{\alpha_t}{|Z_t|}.
 \label{eq:schedule-aggregates}
\end{equation}
Let $\mathbf Z_T:=(Z_0,\ldots,Z_{T-1})$ denote the mini-batch trajectory
sampled in Algorithm~\ref{alg:smg}. We write $\E_{\mathbf Z_T}$ for expectation over the mini-batch sampling, and $\E_{\mathbf Z_T,\tau}$ for joint expectation over the mini-batch sampling and the independent output-index draw.

\begin{theorem}[Convergence under general stepsize and batch-size schedules]
\label{thm:general-smg}
Let the stepsizes and batch sizes be deterministic inputs to Algorithm~\ref{alg:smg}, with $0<\alpha_t\leq 1/(2L)$. Let $\tau$ be sampled independently of $\mathbf Z_T$ with $\Pr(\tau=t)=\alpha_t/A_T$. Under Assumptions~\ref{ass:smooth-lower}--\ref{ass:oracle}, for every fixed $m\in[M]$ and $T\geq 1$,
\begin{equation}
 \min_{0\leq t\leq T-1}\E_{\mathbf Z_T}\big[R_S(x_t)\big]
 \leq
 \frac{1}{A_T}\sum_{t=0}^{T-1}\alpha_t\E_{\mathbf Z_T}\big[R_S(x_t)\big]
 =\E_{\mathbf Z_T,\tau}\big[R_S(x_\tau)\big]
 \leq
 \frac{8\big(f_{S,m}(x_0)-f_{S,m}^{\inf}\big)+(4+2M)\sigma^2V_T}{A_T}.
 \label{eq:general-rate}
\end{equation}
Because \eqref{eq:general-rate} holds for every $m$, its right-hand side may be minimized over $m\in[M]$. In particular, the stepsize-weighted average $\E_{\mathbf Z_T,\tau}[R_S(x_\tau)]$ converges to zero whenever $A_T\to\infty$ and $V_T/A_T\to 0$.
\end{theorem}

\begin{proof}
Throughout the proof, we abbreviate $\E:=\E_{\mathbf Z_T}$. Fix $m\in[M]$ and write $\xi_{m,t}:=q_{t,m}-\nabla f_{S,m}(x_t)$. By Lemma~\ref{lem:stochastic-common-descent}, the following chain holds:
\begin{equation}
 \nabla f_{S,m}(x_t)^\top d_t
 =q_{t,m}^\top d_t-\xi_{m,t}^\top d_t
 \leq -\|d_t\|^2-\xi_{m,t}^\top d_t
 \leq - \frac{1}{2} \|d_t\|^2 + \frac{1}{2} \|\xi_{m,t}\|^2.
 \label{eq:descent-direction-decomposition}
\end{equation}
Here, the second inequality in \eqref{eq:descent-direction-decomposition} follows from Young's inequality. Combining $L$-smoothness, \eqref{eq:smg-update}, \eqref{eq:descent-direction-decomposition}, and the stepsize bound gives
\begin{align}
 f_{S,m}(x_{t+1})
 &\leq f_{S,m}(x_t)+\alpha_t\nabla f_{S,m}(x_t)^\top d_t
       +\frac{L\alpha_t^2}{2}\|d_t\|^2
 &&\text{by $L$-smoothness and \eqref{eq:smg-update}}
 \label{eq:descent-lemma}\\
 &\leq f_{S,m}(x_t)-\frac{\alpha_t}{2}\left(1-L\alpha_t\right)\|d_t\|^2
       +\frac{\alpha_t}{2}\|\xi_{m,t}\|^2
 &&\text{by~\eqref{eq:descent-direction-decomposition}}
 \label{eq:smoothness-descent}\\
 &\leq f_{S,m}(x_t)-\frac{\alpha_t}{4}\|d_t\|^2
       +\frac{\alpha_t}{2}\|\xi_{m,t}\|^2
 &&\text{since $L\alpha_t\leq 1/2$}
 \label{eq:one-step-descent}\\
  &\leq f_{S,m}(x_t)-\frac{\alpha_t}{8}R_S(x_t)
       +\frac{\alpha_t}{2}\|\xi_{m,t}\|^2
       +\frac{\alpha_t}{4}\|Q_t-\nabla F_S(x_t)\|_F^2
  &&\text{by Lemma~\ref{lem:ps-measure-lipschitz}}.
  \label{eq:ps-measure-continuity}
\end{align}
To obtain the last inequality in \eqref{eq:ps-measure-continuity}, write $E_t:=Q_t-\nabla F_S(x_t)$. Lemma~\ref{lem:ps-measure-lipschitz} and $(a+b)^2\leq 2a^2+2b^2$ then yield
\begin{align}
 \sqrt{R_S(x_t)}
 &=r\bigl(\nabla F_S(x_t)\bigr)
 \leq r(Q_t)+\|E_t\|_F
 =\|d_t\|+\|E_t\|_F, \nonumber\\
 R_S(x_t)
 &\leq 2\|d_t\|^2+2\|E_t\|_F^2
 \quad\Longrightarrow\quad
 -\frac{\alpha_t}{4}\|d_t\|^2
 \leq -\frac{\alpha_t}{8}R_S(x_t)
 +\frac{\alpha_t}{4}\|E_t\|_F^2.
\label{eq:ps-measure-continuity-unpacked}
\end{align}
Taking conditional expectations in \eqref{eq:ps-measure-continuity} yields the one-step inequality
\begin{align}
 \alpha_t R_S(x_t)
 &\leq 8\Big(f_{S,m}(x_t)-\E\big[f_{S,m}(x_{t+1})\mid\cF_t\big]\Big)
 +4\alpha_t\E\big[\|\xi_{m,t}\|^2\mid\cF_t\big]
 +2\alpha_t\E\big[\|Q_t-\nabla F_S(x_t)\|_F^2\mid\cF_t\big] \notag\\
 &\leq 8\Big(f_{S,m}(x_t)-\E\big[f_{S,m}(x_{t+1})\mid\cF_t\big]\Big)
 +\frac{(4+2M)\alpha_t\sigma^2}{|Z_t|}.
 \label{eq:conditional-ps-measure-descent}
\end{align}
Here, the second inequality follows from the conditional variance
bound in Assumption~\ref{ass:oracle}, which gives
$\E[\|Q_t-\nabla F_S(x_t)\|_F^2\mid\cF_t]\leq M\sigma^2/|Z_t|$.
Each $f_{S,m}$ is bounded below by Assumption~\ref{ass:smooth-lower}. Moreover, induction using \eqref{eq:one-step-descent} and the conditional variance bound in Assumption~\ref{ass:oracle} shows that $\E[f_{S,m}(x_t)]$ is finite for every $t\leq T$. All expectations below are therefore well-defined and finite. Taking total expectations in \eqref{eq:conditional-ps-measure-descent} and summing over $t$ gives
\begin{align}
 \sum_{t=0}^{T-1}\alpha_t\E\big[R_S(x_t)\big]
 &\leq 8\Big(f_{S,m}(x_0)-\E\big[f_{S,m}(x_T)\big]\Big)
 +(4+2M)\sigma^2V_T
 \label{eq:ps-measure-sum-before-lower-bound}\\
 &\leq 8\big(f_{S,m}(x_0)-f_{S,m}^{\inf}\big)+(4+2M)\sigma^2V_T.
 \label{eq:ps-measure-sum}
\end{align}
Conditioning on $\mathbf Z_T$ and using the independence of $\tau$ give
$\E_{\tau}[R_S(x_\tau)\mid\mathbf Z_T]=A_T^{-1}\sum_{t=0}^{T-1}\alpha_tR_S(x_t)$. Taking expectation over $\mathbf Z_T$ establishes the equality below. Dividing \eqref{eq:ps-measure-sum} by $A_T$ and using the minimum--average inequality then give
\begin{equation}
 \min_{0\leq t\leq T-1}\E_{\mathbf Z_T}\big[R_S(x_t)\big]
 \leq
 \frac{1}{A_T}\sum_{t=0}^{T-1}\alpha_t\E_{\mathbf Z_T}\big[R_S(x_t)\big]
 =\E_{\mathbf Z_T,\tau}\big[R_S(x_\tau)\big]
 \leq
 \frac{8\big(f_{S,m}(x_0)-f_{S,m}^{\inf}\big)+(4+2M)\sigma^2V_T}{A_T}.
 \label{eq:rate-conclusion}
\end{equation}
This proves \eqref{eq:general-rate}. The final convergence statement follows directly when $A_T\to\infty$ and $V_T/A_T\to 0$.
\end{proof}

Because $R_S$ involves the full-batch gradient matrix $\nabla F_S$, Algorithm~\ref{alg:smg} may not be able to identify the minimizing iterate in \eqref{eq:general-rate}, and the terminal iterate $x_T$ itself carries no guarantee. Algorithm~\ref{alg:smg} therefore uses a standard randomized output: it returns $x_\tau$ with $\tau\in\{0,\ldots,T-1\}$ drawn independently of the trajectory with probabilities proportional to $\alpha_t$. Then $\E_{\mathbf Z_T,\tau}[R_S(x_\tau)]$ equals the stepsize-weighted average $A_T^{-1}\sum_{t=0}^{T-1}\alpha_t\,\E_{\mathbf Z_T}[R_S(x_t)]$ and is bounded by the right-hand side of \eqref{eq:general-rate}. Under the constant stepsizes of Corollary~\ref{cor:increasing-batches}, $\tau$ is uniform.

\begin{corollary}[Convergence under diverging mini-batch sizes]
\label{cor:increasing-batches}
Under the assumptions of Theorem~\ref{thm:general-smg}, suppose in addition that $\alpha_t=\alpha\in(0,1/(2L)]$ for all $t$ and $|Z_t|\to\infty$. Then, for every $m\in[M]$,
\begin{align}
 \E_{\mathbf Z_T}\left[\min_{0\leq t\leq T-1}R_S(x_t)\right]
 &\leq \min_{0\leq t\leq T-1}\E_{\mathbf Z_T}\big[R_S(x_t)\big]  
 \leq \frac{1}{T}\sum_{t=0}^{T-1}\E_{\mathbf Z_T}\big[R_S(x_t)\big]
  =\E_{\mathbf Z_T,\tau}\big[R_S(x_\tau)\big] \notag\\
 &\leq \frac{8\big(f_{S,m}(x_0)-f_{S,m}^{\inf}\big)}{\alpha T}
 +(4+2M)\sigma^2\frac{1}{T}\sum_{t=0}^{T-1}\frac{1}{|Z_t|}
 \longrightarrow 0.
 \label{eq:increasing-batch-rate}
\end{align}
If $|Z_t|\geq b(t+1)$ for some $b>0$, then \eqref{eq:increasing-batch-rate} specializes to
\begin{align}
 \E_{\mathbf Z_T}\left[\min_{0\leq t\leq T-1}R_S(x_t)\right]
 &\leq \min_{0\leq t\leq T-1}\E_{\mathbf Z_T}\big[R_S(x_t)\big]  
 \leq \frac{1}{T}\sum_{t=0}^{T-1}\E_{\mathbf Z_T}\big[R_S(x_t)\big]
  =\E_{\mathbf Z_T,\tau}\big[R_S(x_\tau)\big] \notag\\
 &\leq \frac{8\big(f_{S,m}(x_0)-f_{S,m}^{\inf}\big)}{\alpha T}
 +\frac{(4+2M)\sigma^2\big(1+\log T\big)}{bT}
 =\widetilde O\left(\frac{M}{T}\right).
 \label{eq:linear-batch-rate-with-M}
\end{align}
\end{corollary}

\begin{proof}
For constant stepsizes, $A_T=\alpha T$ and $V_T=\alpha\sum_{t=0}^{T-1}|Z_t|^{-1}$. The first inequality in \eqref{eq:increasing-batch-rate} follows because $\min_{0\leq t\leq T-1}R_S(x_t)\leq R_S(x_t)$ almost surely for every $t$, whereas the remaining relations follow from the minimum--average inequality, the uniform distribution of $\tau$, and Theorem~\ref{thm:general-smg}. Because $|Z_t|^{-1}\to 0$, its Ces\`aro mean also converges to zero. Under linear batch growth,
\begin{equation}
 \frac{1}{T}\sum_{t=0}^{T-1}\frac{1}{|Z_t|}
 \leq\frac{1}{bT}\sum_{t=0}^{T-1}\frac{1}{t+1}
 \leq\frac{1+\log T}{bT}.
 \label{eq:linear-batch-specialization}
\end{equation}
Substitution into \eqref{eq:increasing-batch-rate} gives \eqref{eq:linear-batch-rate-with-M} and completes the proof.
\end{proof}

\begin{remark}[Why the rate improves over the prior SMG analysis]
\label{rem:rate-improvement}
For linearly growing mini-batches and $\alpha=\Theta(T^{-1/2})$, \citet[Eq.~(C.23)]{chen2024threeway} obtain an $\widetilde O(T^{-1/4})$ bound for the squared empirical PS measure. Applying Jensen's inequality yields the $\widetilde O(T^{-1/8})$ bound for the unsquared measure reported in their Theorem~8. Their proof uses the $(1/2)$-H{\"o}lder continuity of $Q\mapsto d_Q$ and bounded gradients (their Assumption~3) when relating the stochastic MGDA direction to the empirical PS measure. Our analysis does not require bounded gradients.

In contrast, our proof does not use this $(1/2)$-H{\"o}lder continuity of $Q\mapsto d_Q$.
Equation~\eqref{eq:ps-measure-continuity} instead uses the Lipschitz continuity of the PS measure $Q\mapsto r(Q)=\|d_Q\|$ from Lemma~\ref{lem:ps-measure-lipschitz}. The mini-batch error therefore enters through its second moment, rather than through a fourth root. With constant admissible stepsizes and linearly growing batches $|Z_t|\geq b(t+1)$, Theorem~\ref{thm:general-smg} gives $\min_{0\leq t<T}\E[R_S(x_t)]=\widetilde O(T^{-1})$.
Below we provide a side-by-side comparison of the proof structures.
\end{remark}

\medskip
\noindent\textbf{Side-by-side comparison of the proof structures.}
Recall that $Q_t=\nabla F_{Z_t}(x_t)$, $d_t=d_{Q_t}$, and $E_t=Q_t-\nabla F_S(x_t)$, with $m$-th column $\xi_{m,t}=q_{t,m}-\nabla f_{S,m}(x_t)$, where $q_{t,m}=\nabla f_{Z_t,m}(x_t)$. Thus, $\|E_t\|_F^2=\sum_{m=1}^M\|\xi_{m,t}\|^2$ and $\|\xi_{m,t}\|\leq\|E_t\|_F$. Also recall that $r(Q)=\|d_Q\|$ and $R_S(x_t)=r(\nabla F_S(x_t))^2$.

This comparison is schematic, with constants and lower-order terms suppressed. The right panel includes the pointwise moment refinement in part~1 of Remark~\ref{rem:chen-proof-improvement}, which removes the logarithmic factor from the original bound. Throughout the comparison, $a\lesssim b$ means $a\leq Cb$ for a positive constant $C$ independent of $t$, $T$, $\alpha_t$, and $|Z_t|$, but possibly dependent on fixed problem parameters. We use $\lesssim$ below to emphasize the proof mechanisms. Both proofs apply $L$-smoothness to the same update, 
and use the same Lyapunov energy,
giving
\begin{align}\label{eq:m_obj_smooth}
f_{S,m}(x_{t+1})
&\leq f_{S,m}(x_t)+
\underbrace{\alpha_t \nabla f_{S,m}(x_t)^\top d_t
+\frac{L\alpha_t^2}{2}\|d_t\|^2}_{\text{term $(*)$, key difference}}.
\end{align}
They differ in how they control the term $(*)$ before telescoping.
To bound the term $(*)$, 
both proofs first bound the term $\nabla f_{S,m}(x_t)^\top d_t$
and then combine this estimate with the smoothness term $\frac{L\alpha_t^2}{2}\|d_t\|^2$.
\begin{center}
\begin{minipage}[t]{0.485\textwidth}
\vspace{0pt}
\begin{tcolorbox}[
  enhanced,
  colback=green!3!white,
  colframe=boxbordergreen,
  colbacktitle=boxbordergreen,
  coltext=black,
  coltitle=green!3!white,
  fonttitle=\bfseries,
  title={Present proof (Theorem~{\colorlet{boxbordergreen}{green!3!white}\ref{thm:general-smg}})},
  boxrule=0.6pt,
  arc=1mm,
  left=1.5mm,
  right=1.5mm,
  top=1mm,
  bottom=1mm
]
\footnotesize
\textbf{1. Bound term $(*)$ in~\eqref{eq:m_obj_smooth}.}

\smallskip
\noindent{\setlength{\fboxsep}{1pt}%
\colorbox{green!15!white}{%
\parbox{\dimexpr\linewidth-2\fboxsep\relax}{%
\textbf{Key step: decompose $\nabla f_{S,m}(x_t)$ 
into $q_{t,m}-\xi_{m,t}$
and bound the PS measure difference.}
Common descent and the error decomposition yield
\begin{align}
\nabla f_{S,m}(x_t)^\top d_t
&=q_{t,m}^\top d_t-\xi_{m,t}^\top d_t \notag\\
&\leq -\|d_t\|^2-\xi_{m,t}^\top d_t \notag\\
&\lesssim - \|d_t\|^2+ \|\xi_{m,t}\|^2.
\end{align}
The Lipschitz continuity of the PS measure and the stepsize bound then give
\begin{align}
\alpha_t \nabla f_{S,m}(x_t)^\top d_t +\frac{L\alpha_t^2}{2}\|d_t\|^2
&\lesssim - \alpha_t R_S(x_t) + \alpha_t \|E_t\|_F^2.
\end{align}
This avoids comparing the stochastic and full-batch directions, and hence avoids the $(1/2)$-H{\"o}lder continuity term.
}}}
\smallskip

\textbf{2. Telescope.}
Summing gives
\begin{align}
\sum_{t=0}^{T-1} \alpha_t \E[R_S(x_t)]
\lesssim
1 + \sum_{t=0}^{T-1} \alpha_t\E[\|E_t\|_F^{2}].
\end{align}
\smallskip

\textbf{3. Insert the variance and the schedule.}
The matrix-error variance bound gives
\begin{align}
\min_{0\leq t\leq T-1} \E[R_S(x_t)]
=O\Big( \frac{1}{A_T} +
\frac{\sum_{t=0}^{T-1} \alpha_t|Z_t|^{-1}}{A_T} \Big).
\end{align}

\smallskip
\textbf{4. Specialize the rate.}
With linear batch growth $|Z_t|\geq b(t+1)$ 
and stepsize $\alpha_t=\alpha=\Theta(1)$,
$\min_{0\leq t<T}\E[R_S(x_t)]=\widetilde O(T^{-1})$.
\end{tcolorbox}
\end{minipage}\hfill
\begin{minipage}[t]{0.485\textwidth}
\vspace{0pt}
\begin{tcolorbox}[
  enhanced,
  colback=green!3!white,
  colframe=boxbordergreen,
  colbacktitle=boxbordergreen,
  coltext=black,
  coltitle=green!3!white,
  fonttitle=\bfseries,
  title={Proof route of {\colorlet{boxbordergreen}{green!3!white}\protect\citet{chen2024threeway}}, Appendix~C.2},
  boxrule=0.6pt,
  arc=1mm,
  left=1.5mm,
  right=1.5mm,
  top=1mm,
  bottom=1mm
]
\footnotesize
\textbf{1. Bound term $(*)$ in~\eqref{eq:m_obj_smooth}.}

\smallskip
\noindent{\setlength{\fboxsep}{1pt}%
\colorbox{green!15!white}{%
\parbox{\dimexpr\linewidth-2\fboxsep\relax}{%
\textbf{Key step: decompose $d_t$ into
$d_{\nabla F_S(x_t)} + (d_t-d_{\nabla F_S(x_t)})$
and bound the MGDA direction difference.}
Common descent and the $(1/2)$-H{\"o}lder continuity of $d_Q$ yield
\begin{align}
&\nabla f_{S,m}(x_t)^\top d_t \notag\\
&=\nabla f_{S,m}(x_t)^\top d_{\nabla F_S(x_t)}
 +\nabla f_{S,m}(x_t)^\top\bigl(d_t-d_{\nabla F_S(x_t)}\bigr) \notag\\
&\leq -R_S(x_t)+\|\nabla f_{S,m}(x_t)\|\,
  \|d_t-d_{\nabla F_S(x_t)}\| \notag\\
&\lesssim -R_S(x_t)+\|E_t\|_F^{1/2}.
\end{align}
Combining bounded gradients with the above inequality then gives
\begin{align}
&\alpha_t \nabla f_{S,m}(x_t)^\top d_t +\frac{L\alpha_t^2}{2}\|d_t\|^2 \notag\\
&\lesssim - \alpha_t R_S(x_t) + \alpha_t \|E_t\|_F^{1/2} + \alpha_t^2.
\end{align}
}}}
\smallskip

\textbf{2. Telescope.}
Summing gives
\begin{align}
\sum_{t=0}^{T-1} \alpha_t \E[R_S(x_t)]
\lesssim
1 + \sum_{t=0}^{T-1} \Big(\alpha_t\E[\|E_t\|_F^{1/2}]
+ \alpha_t^2 \Big).
\end{align}

\smallskip
\textbf{3. Insert the variance and the schedule.}
Jensen's inequality yields the fourth-root term.
\begin{align}
\min_{0\leq t\leq T-1} \E[R_S(x_t)]
&=O\left(\frac{1+\sum_{t=0}^{T-1}(\alpha_t|Z_t|^{-\frac{1}{4}}+\alpha_t^2)}{A_T}\right).
\end{align}

\smallskip
\textbf{4. Specialize the rate.}
With linear batch growth $|Z_t|\geq b(t+1)$ 
and stepsize $\alpha_t=\alpha=\Theta(T^{-1/2})$,
$\min_{0\leq t<T}\E[R_S(x_t)]=O(T^{-1/4})$.
\end{tcolorbox}
\end{minipage}
\end{center}

\begin{remark}[Two refinements of the SMG analysis in {\citet{chen2024threeway}}]
\label{rem:chen-proof-improvement}
\textbf{1) Error sum.}
\citet{chen2024threeway} average the second-moment error before taking a quarter power, introducing a logarithm under linear batch growth. 
Applying the moment bound pointwise instead removes the logarithmic factor and yields a convergence rate of $O(T^{-1/4})$. 
\textbf{2) Constant stepsize.}
Retaining $\alpha_t^2\|E_t\|_F$ in the bound for $\frac{L\alpha_t^2}{2}\|d_t\|^2$, rather than replacing it by $O(\alpha_t^2)$, allows the $\alpha_t^2R_S(x_t)$ term to be absorbed into the descent term when the constant stepsize $\alpha_t = \alpha=\Theta(1)$ is sufficiently small.
This refinement does not change the final convergence rate $O(T^{-1/4})$
because the stepsize-dependent terms are not dominant.
\end{remark}
Appendix~\ref{app:momentum-exact-mgda} extends this proof technique to MoCo and MoCo+ with exact MGDA updates. Theorem~\ref{thm:exact-moco} gives an $O(T^{-1/2})$ squared empirical PS rate for MoCo, improving the $O(T^{-1/8})$ bound obtained from equations~(C.21) and~(C.24) of \citet{chen2024threeway}. For MoCo+, Theorem~\ref{thm:exact-mocop} gives $\widetilde O(T^{-2/3})$, matching the STORM rate and, up to logarithmic factors, the exponent of the prior MoCo+ rate~\citep{cutkosky2019momentum,fernando2024vr}.

\section{Related work}
\label{sec:related-work}

\paragraph{Stochastic MGDA.}
Stochastic MGDA is a variant of MGDA 
that uses stochastic gradients instead of deterministic ones. 
\citet{mercier2018smgda} first establish asymptotic mean-square and almost-sure convergence. 
Even unbiased stochastic gradients can yield a biased MGDA direction because the weights depend on the sampled gradients~\citep{zhou2022smgm,fernando2023moco}.
To reduce this bias, \citet{liu2024smg} increase the batch size during optimization, but their convergence analysis assumes Lipschitz continuity of the optimal weights $\lambda^\star_Q$ with respect to the gradient matrix $Q$.
This assumption fails in general, as shown by \citet[Proposition~2]{zhou2022smgm}.
Momentum-based methods with one-step inexact MGDA updates bypass this assumption and reduce the bias: correlation-reduced stochastic multi-objective gradient manipulation (CR-MOGM) averages the weights~\citep{zhou2022smgm}, whereas stochastic multi-objective gradient correction (MoCo) tracks the gradients~\citep{fernando2023moco}.
Other methods use independent double sampling with single-loop updates, including multi-objective gradient with double sampling (MoDo)~\citep{chen2024threeway} and stochastic direction-oriented multi-objective gradient descent (SDMGrad)~\citep{xiao2023sdmgrad}. The adaptive sampling stochastic multigradient algorithm (ASSMG) uses adaptive batches under a Lipschitz-weight assumption~\citep{zhao2024adaptive}, whereas the variance-reduced extension MoCo+ uses a stochastic recursive momentum (STORM) estimator~\citep{cutkosky2019momentum,fernando2024vr}. Further variants include stochastic generalized smooth multi-objective gradient descent (SGSMGrad)~\citep{zhang2025generalized} and periodic stochastic multi-gradient descent (PSMGD)~\citep{xu2025psmgd}. \citet{chen2024threeway} also refine the SMG proof and overcome the limitation of the Lipschitz-weight assumption by replacing it with a proof of $(1/2)$-H\"{o}lder continuity for the MGDA direction, while \citet{huangchen2026more} identify stronger regularity conditions for Lipschitz direction continuity 
and develop a stochastic multi-objective regularity-aware (MoRe) method. 
The present analysis instead uses the Lipschitz continuity of the PS measure in Lemma~\ref{lem:ps-measure-lipschitz} 
to obtain a better rate than the growing-batch analysis of \citet{chen2024threeway}, through an argument distinct from the analyses above.
Table~\ref{tab:stochastic-mgda-comparison}
in Appendix~\ref{app:stochastic-mgda-comparison}
summarizes the sampling regimes, variance-control mechanisms, assumptions, and squared Pareto-stationarity rates of these works.

Other works consider different finite-sum settings: joint gradient balancing for data ordering (JoGBa) uses a without-replacement ordering based on joint online vector balancing~\citep{yangkwok2025jogba}, whereas the finite-sum method STIMULUS periodically invokes a full-batch gradient oracle to reset a recursive estimator~\citep{liu2025stimulus}. 
For a complementary comparison of stochastic multi-objective methods
and other relevant settings, see the survey by \citet{chen2025gradientbased}.

\paragraph{Optimization proofs assisted by large language models.}
Recent work has begun to document substantive uses of large language models (LLMs) in optimization and mathematical proof discovery. \citet{jangryu2025nag} report that ChatGPT assisted in discovering a proof that Nesterov's accelerated gradient method converges pointwise, and they describe the elicitation process. \citet{salim2025convex} studies a GPT-5-Pro-assisted proof of a convex-analysis lemma concerning the biconjugation operator. That case study reports both useful intermediate arguments and subtle model errors that required expert correction.
\citet{ma2026bdrs} reports that ChatGPT~5.5 generated a convergence proof for Bregman Douglas--Rachford splitting as a matrix-scaling method, with author verification. \citet{orabona2026omwu} identifies a ChatGPT-assisted boundary argument as the key new step in a last-iterate convergence proof for optimistic multiplicative weights. These examples illustrate how 
LLMs can accelerate optimization research and how
LLM disclosure can identify the assisted component while retaining human responsibility for the final mathematical statement.

Unlike the works above, in which the authors deliberately prompted an LLM on an explicitly posed open problem, the proof here emerged accidentally: the author was preparing course homework solutions on a question for which the best previously known rate was $\widetilde O(T^{-1/4})$~\citep{chen2024threeway}, and this note improves the rate to $\widetilde O(T^{-1})$.
Given a precise problem statement, surrounding context, and supporting lemmas, ChatGPT~5.4 TE generated the initial proof strategy for vanilla SMG. The author then checked the mathematical steps and reorganized the argument.
This note addresses a tractable question not resolved by the prior analyses summarized in Appendix~\ref{app:stochastic-mgda-comparison}.
Appendix~\ref{app:proof-discovery} further documents the proof-generation process and compares the successful and unsuccessful routes found in the sampled homework submissions.

\paragraph{Other LLM-assisted scientific research.}
LLMs have also been used in workflows that produce independently checkable optimization certificates. \citet{kimpilanci2026dual} use a coding agent to propose convex relaxations and a theory agent to test their validity and search for counterexamples, with the reported lower bounds certified by dual-feasible points checked using rigorous interval arithmetic. More broadly, the GPT-5 science-acceleration report collects mathematical case studies in which human experts verify AI-assisted results and emphasizes that models can make confident errors~\citep{bubeck2025science}.

\section{Conclusion}
\label{sec:conclusion}

This note establishes a finite-time convergence rate for vanilla SMG on smooth nonconvex stochastic MOO. Under a constant admissible stepsize and linearly growing mini-batches, the expected squared empirical PS measure converges at rate $\widetilde O(T^{-1})$, with $T$ denoting the number of iterations. The key rate improvement rests on the Lipschitz continuity of the PS measure.

The proof was discovered while the author was preparing homework for a course: ChatGPT~5.4 TE generated the initial proof strategy from the documented prompt, and the author verified, corrected, and reorganized it. The author remains responsible for all claims and any remaining errors.

The author hopes this note sheds light on how 
AI can shape future research and education.

\section*{Acknowledgments}
This work began during the development of the graduate Electrical and Computer Engineering course ``ECE-406 Introduction to Multi-objective Machine Learning'' at the University of Rochester in Spring 2026. The author thanks the University and the Electrical and Computer Engineering Graduate Education Committee for supporting AI-assisted course development and thanks the participating students for documenting their LLM prompts and responses.

\paragraph{Use of AI.}
ChatGPT~5.4 TE generated the initial proof strategy for Theorem~\ref{thm:general-smg} in response to the author-created prompt reproduced in Appendix~\ref{app:proof-discovery}. ChatGPT~5.6 Sol assisted in deriving the momentum proofs in Appendix~\ref{app:momentum-exact-mgda} and, together with Claude Fable~5, in literature searches and presentation. The author independently verified, corrected, and reorganized the proofs, checked all cited sources, made the final revisions, and takes full responsibility for the entire manuscript. No generative AI system is listed as an author or cited as a source for a mathematical claim.

\appendix
\section{Stochastic MGDA prior-work comparison}
\label{app:stochastic-mgda-comparison}

\normalsize
This appendix complements the related-work discussion with a comparison of representative stochastic MGDA methods. To keep the comparison concise, Table~\ref{tab:stochastic-mgda-comparison} reports only the best convergence rate established for each method in the corresponding paper, rather than all rates obtained under different settings. Additional assumptions and cost qualifications required for the displayed rates are stated in the table notes. The table also summarizes the batch regimes, continuity properties, and mechanisms for controlling stochastic error.
This summary is based on an extensive search for closely related work.
Among the works surveyed below, none establishes an
$\widetilde O(T^{-1})$ rate for the squared empirical PS measure of
vanilla SMG under linearly growing mini-batches or obtains such a rate
from the Lipschitz continuity of the PS measure.
This justifies the novelty of this note.

\begingroup
\normalsize
\setlength{\tabcolsep}{3pt}
\uniformtableformat
\setlength{\LTcapwidth}{\textwidth}
\begin{longtable}{>{\raggedright\arraybackslash}m{0.16\textwidth}|>{\centering\arraybackslash}m{0.085\textwidth}>{\centering\arraybackslash}m{0.155\textwidth}>{\raggedright\arraybackslash}m{0.265\textwidth}>{\centering\arraybackslash}m{0.245\textwidth}}
\caption{Comparison with stochastic MGDA methods. The rate column reports the convergence rate of the squared PS measure in each paper's own setting, including whether the objective is a population or an empirical objective and whether the guarantee is in expectation or with high probability. The symbol $R_S$ denotes the empirical measure used in this note. The notation $\widetilde O(\cdot)$ hides logarithmic factors. The displayed rates suppress dependence on the number of objectives $M$ and may rely on additional per-paper assumptions or cost qualifications stated in the table notes. Here $T$ denotes outer iterations. In the ``Used continuity'' column, ``Assumed'' denotes an additional hypothesis, ``Proved'' denotes a property established in the cited analysis, and ``---'' means that no such continuity property is invoked directly. The corresponding methods typically bypass this challenge by using independent double sampling, regularization, or a one-step inexact MGDA update in the inner loop.
Results are listed chronologically based on their availability online 
(e.g., preprint versions).}
\label{tab:stochastic-mgda-comparison}\\[6pt]
\toprule
\textbf{Method} & \textbf{Batch} & \textbf{Used continuity} & \textbf{Variance/bias control} & \textbf{Squared PS measure convergence rate} \\
\midrule
\endfirsthead
\toprule
\textbf{Method} & \textbf{Batch} & \textbf{Used continuity} & \textbf{Variance/bias control} & \textbf{Squared PS measure convergence rate} \\
\midrule
\endhead
\endfoot
\bottomrule
\endlastfoot
SMG (prior) \citep{liu2024smg}
  & $\Theta(t+1)$ & Assumed Lipschitz $\lambda^\star_Q$ & Growing batches  
  & convex rates only, no nonconvex PS rate \\
\midrule
CR-MOGM \citep{zhou2022smgm}
  & $\Theta(1)$ & --- & Momentum on weights and one inner stochastic gradient descent (SGD) step
  & $O(T^{-1/4})$ \\
\midrule
MoCo\textsuperscript{$\dagger$} \citep{fernando2023moco}
  & $\Theta(1)$ & --- & Momentum on gradients, regularization,
  and one inner SGD step
  & $O(T^{-1/2})$ \\
\midrule
MoDo \citep{chen2023threeway}
  & $\Theta(1)$ & --- & Double sampling and one inner SGD step
  & $O(T^{-1/2})$ \\
\midrule
SDMGrad \citep{xiao2023sdmgrad}
  & $\Theta(1)$ & --- & Double sampling and one inner SGD step
  & $O(T^{-1/2})$ \\
\midrule
SMG (reanalysis) \citep{chen2024threeway}
  & $\Theta(t+1)$ & \makecell[l]{Proved $\frac{1}{2}$-H\"older\\ direction $d_Q$} & Growing batches
  & $\widetilde O(T^{-1/4})$ \\
\midrule
ASSMG\textsuperscript{$\ddagger$} \citep{zhao2024adaptive}
  & Adaptive & Assumed Lipschitz $\lambda^\star_Q$ & Adaptive batches with a relative-error batch test
  & $O(T^{-1})$ \\
\midrule
MoCo+ \citep{fernando2024vr}
  & $\Theta(1)$ & --- & STORM-type momentum on gradients and one inner SGD step
  & $O(T^{-2/3})$ \\
\midrule
SGSMGrad \citep{zhang2025generalized}
  & $\Theta(1)$ & --- & Double sampling and high-probability control
  & $O(T^{-1/2})$ \\
\midrule
PSMGD \citep{xu2025psmgd}
  & $\Theta(1)$ & --- & Periodic subprograms and blended stale weights
  & $O(T^{-1/2})$ \\
\midrule
MoRe \citep{huangchen2026more}
  & $\Theta(t+1)$ & \makecell[l]{Proved Lipschitz \\ regular $d_Q$} & Regularity test and scalarization fallback
  & $\widetilde O(T^{-1/2})$ \\
\midrule
\rowcolor{green!3!white}
\textbf{SMG (reanalysis, this work)}\par
Theorem~\ref{thm:general-smg}
  & $\Theta(t+1)$ & \makecell[l]{Proved Lipschitz \\ norm $\|d_Q\|$} 
  & Growing batches
  & $\widetilde O(T^{-1})$ \\
\midrule
\rowcolor{green!3!white}
\textbf{MoCo (exact MGDA, 
reanalysis, this work)}\par
Theorem~\ref{thm:exact-moco}
  & $\Theta(1)$ & \makecell[l]{Proved Lipschitz \\ norm $\|d_Q\|$}
  & Exponential-type momentum on gradients
  & $O(T^{-1/2})$ \\
\midrule
\rowcolor{green!3!white}
\textbf{MoCo+ (exact MGDA, this work)}\par
Theorem~\ref{thm:exact-mocop}
  & $\Theta(1)$ & \makecell[l]{Proved Lipschitz \\ norm $\|d_Q\|$}
  & STORM-type momentum on gradients
  & $\widetilde O(T^{-2/3})$ \\
\end{longtable}
\vspace{-0.5\baselineskip}
\noindent\footnotesize\raggedright
\textsuperscript{$\dagger$}The $O(T^{-1/2})$ rate in \citet[Theorem~3]{fernando2023moco} assumes bounded gradients and objective values along the trajectory, together with a gradient-bias condition that is automatic for unbiased oracles. Their Theorem~1 gives $O(T^{-1/12})$ under a weaker bias condition and without bounded objective values.\par
\textsuperscript{$\ddagger$}The method uses adaptive batches, and no bound on total stochastic-gradient evaluations is provided.\par
The SMG reanalysis of \citet[Eq.~(C.23)]{chen2024threeway} uses bounded gradients; the results proved in this note do not impose that assumption.\par
\endgroup

\section{Momentum-based stochastic MGDA with exact MGDA updates}
\label{app:momentum-exact-mgda}

This appendix extends the convergence analysis of Theorem~\ref{thm:general-smg}
to variants of MoCo and MoCo+ with exact MGDA updates~\citep{fernando2023moco,fernando2024vr}.
With assistance from ChatGPT~5.6 Sol, the author combined the proof technique of Theorem~\ref{thm:general-smg} with the single-objective momentum analyses of \citet{guo2021adam} for MoCo and \citet{cutkosky2019momentum} for MoCo+.
Both methods use a fixed batch size $|Z_t|\equiv|Z_0|\geq1$ and solve the unregularized MGDA subproblem exactly at each iteration.
For each objective $m\in[M]$, let $y_{t,m}\in\R^p$ denote the momentum
estimate of $\nabla f_{S,m}(x_t)$ at iteration $t$.
Collect these estimates in the matrix $Y_t\in\R^{p\times M}$,
and define the estimation error $E_t$ (with $Y_t$ in place of $Q_t$), its mean squared norm $e_t$, and the variance bound
$\bar\sigma^2$ by
\begin{equation}
Y_t=[y_{t,1},\ldots,y_{t,M}],\qquad
E_t:=Y_t-\nabla F_S(x_t),\qquad
e_t:=\E\big[\|E_t\|_F^2\big],
\qquad
\bar\sigma^2:=\frac{M\sigma^2}{|Z_0|}.
\label{eq:momentum-tracker-notation}
\end{equation}
Both estimators are initialized as $Y_0=\nabla F_{Z_0}(x_0)$ at a deterministic
initial point $x_0$. The MoCo and MoCo+ recursions are given in
\eqref{eq:exact-moco-tracker} and \eqref{eq:exact-mocop-tracker}, respectively.
Throughout this appendix, expectations are taken over the sampled batches
and, where applicable, the independently sampled output index. At iteration $t$, with stepsize
$\alpha_t>0$, both algorithms compute
\begin{subequations}
\begin{align}
\lambda_t&\in\argmin_{\lambda\in\Delta^M}\|Y_t\lambda\|^2,
\label{eq:momentum-exact-weight}\\
d_t&:=-Y_t\lambda_t,
\label{eq:momentum-exact-direction}\\
x_{t+1}&:=x_t+\alpha_t d_t.
\label{eq:momentum-exact-update}
\end{align}
\end{subequations}
Thus, $d_t=d_{Y_t}$ is the exact MGDA direction associated with the gradient
estimator $Y_t$.

All stepsizes and momentum coefficients are deterministic. Each new
batch is sampled uniformly with replacement from $S$, independently of
previously sampled batches. We impose the uniform version of
Assumption~\ref{ass:oracle}: let $\mathcal H$ be the $\sigma$-algebra
generated by the history before a new batch $Z$ of size
$|Z|=|Z_t|=\Theta(1)$ is sampled. For every $\mathcal H$-measurable query point
$x$, almost surely,
\begin{equation}
 \E\big[\nabla F_Z(x)\mid\mathcal H\big]=\nabla F_S(x),
 \qquad
 \E\big[\|\nabla F_Z(x)-\nabla F_S(x)\|_F^2\mid\mathcal H\big]
 \leq\bar\sigma^2.
 \label{eq:momentum-uniform-oracle}
\end{equation}
This condition follows from the uniform per-sample variance bound discussed
after Assumption~\ref{ass:oracle}. Assumption~\ref{ass:smooth-lower} also implies
\begin{equation}
\|\nabla F_S(x)-\nabla F_S(x')\|_F
\leq\sqrt M L\|x-x'\|,
\qquad
\|\nabla F_Z(x)-\nabla F_Z(x')\|_F
\leq\sqrt M L\|x-x'\|.
\label{eq:momentum-matrix-smoothness}
\end{equation}

\begin{corollary}[Descent and stationarity bounds for exact MGDA updates]
\label{lem:momentum-descent-transfer}
Suppose Assumption~\ref{ass:smooth-lower} holds. 
Given an iterate $x_t\in\R^p$ and
a gradient estimator $Y_t\in\R^{p\times M}$, let $\lambda_t$, $d_t$,
and $x_{t+1}$ be computed by the exact MGDA update
\eqref{eq:momentum-exact-weight}--\eqref{eq:momentum-exact-update}, and set
$E_t:=Y_t-\nabla F_S(x_t)$.
If $0<\alpha_t\leq1/(2L)$, then, for every $m\in[M]$,
\begin{align}
f_{S,m}(x_{t+1})
&\leq f_{S,m}(x_t)-\frac{\alpha_t}{4}\|d_t\|^2
+\frac{\alpha_t}{2}\|E_t\|_F^2,
\label{eq:momentum-one-step-descent}\\
R_S(x_t)
&\leq2\|d_t\|^2+2\|E_t\|_F^2.
\label{eq:momentum-residual-transfer}
\end{align}
These bounds hold for each realization of $x_t$ and $Y_t$ and therefore
apply to both MoCo with an exact MGDA update (Algorithm~\ref{alg:exact-moco}) and
MoCo+ with an exact MGDA update (Algorithm~\ref{alg:exact-mocop}) whenever the above stepsize  condition holds.
\end{corollary}

\begin{proof}
For $m\in[M]$, define $e_{t,m}:=y_{t,m}-\nabla f_{S,m}(x_t)$,
the $m$-th column of $E_t$. Applying Lemma~\ref{lem:stochastic-common-descent}
to $Y_t$ gives
\begin{align}
 \nabla f_{S,m}(x_t)^\top d_t
 &\leq-\|d_t\|^2-e_{t,m}^\top d_t
 &&\text{by Lemma~\ref{lem:stochastic-common-descent}} \notag\\
 &\leq-\frac12\|d_t\|^2+\frac12\|e_{t,m}\|^2
 &&\text{by Young's inequality}.
 \label{eq:momentum-common-descent}
\end{align}
Combining \eqref{eq:momentum-common-descent} with $L$-smoothness,
\eqref{eq:momentum-exact-update}, $\|e_{t,m}\|\leq\|E_t\|_F$, and
$L\alpha_t\leq1/2$ proves \eqref{eq:momentum-one-step-descent}.
Lemma~\ref{lem:ps-measure-lipschitz}, with $r(Y_t)=\|d_t\|$, gives
\begin{equation}
 \sqrt{R_S(x_t)}
 =r\bigl(\nabla F_S(x_t)\bigr)
 \leq\|d_t\|+\|E_t\|_F.
 \label{eq:momentum-residual-continuity}
\end{equation}
Squaring \eqref{eq:momentum-residual-continuity} and using
$(a+b)^2\leq2a^2+2b^2$ proves \eqref{eq:momentum-residual-transfer}.
\end{proof}

\begin{keyinsightbox}
\textbf{Outline of the convergence analysis.}
Corollary~\ref{lem:momentum-descent-transfer} bounds the objective decrease
and empirical PS measure in terms of the MGDA direction and estimation
error. Each analysis then controls $e_t$ using the corresponding estimator
recursion. Let $\beta_{t+1}\in(0,1]$ denote the MoCo momentum coefficient
in \eqref{eq:exact-moco-tracker}. The MoCo analysis uses the weight
$\alpha_t/\beta_{t+1}$ to control the contribution of changes in the
empirical gradient. For MoCo+, gradients are evaluated at $x_t$ and
$x_{t+1}$ using the same new batch. A Lyapunov function with estimation
error weighted by the inverse stepsize controls the resulting gradient
difference.
\end{keyinsightbox}

\subsection{MoCo with an exact MGDA update}

MoCo uses an exponentially weighted gradient estimator
\begin{equation}
 Y_{t+1}
 =(1-\beta_{t+1})Y_t
 +\beta_{t+1}\nabla F_{Z_{t+1}}(x_{t+1}),
 \qquad 0<\beta_{t+1}\leq1.
 \label{eq:exact-moco-tracker}
\end{equation}
The order of operations in Algorithm~\ref{alg:exact-moco} ensures that
$x_{t+1}$ is determined before the new batch $Z_{t+1}$ is sampled.

\begin{algorithm}[H]
\small
\caption{MoCo with an exact MGDA update}
\label{alg:exact-moco}
\begin{algorithmic}[1]
\State \textbf{Input:} training dataset $S$, initial point $x_0$, horizon
$T$, batch size $|Z_t|$, stepsizes $\{\alpha_t\}_{t=0}^{T-1}$, and momentum
coefficients $\{\beta_{t+1}\}_{t=0}^{T-1}$.
\State Sample $Z_0$ and set $Y_0=\nabla F_{Z_0}(x_0)$.
\For{$t=0,\ldots,T-1$}
  \State Compute $\lambda_t$ and $d_t$ by
  \eqref{eq:momentum-exact-weight}--\eqref{eq:momentum-exact-direction}.
  \State Update $x_{t+1}$ by \eqref{eq:momentum-exact-update}.
  \State Sample a new batch $Z_{t+1}$.
  \State Update $Y_{t+1}$ by \eqref{eq:exact-moco-tracker}.
\EndFor
\State Sample $\tau\in\{0,\ldots,T-1\}$ independently with
$\Pr(\tau=t)=\alpha_t/\sum_{s=0}^{T-1}\alpha_s$.
\State \textbf{Output:} $x_\tau$.
\end{algorithmic}
\end{algorithm}

\begin{theorem}[Convergence of MoCo with an exact MGDA update]
\label{thm:exact-moco}
Let $\{x_t,Y_t\}_{t=0}^T$ be generated by Algorithm~\ref{alg:exact-moco}
from a deterministic $x_0$, with $Y_0=\nabla F_{Z_0}(x_0)$, fixed batch size
$|Z_t|\equiv|Z_0|\geq1$, and horizon $T\geq1$.
Suppose that Assumptions~\ref{ass:smooth-lower}--\ref{ass:oracle} and
\eqref{eq:momentum-uniform-oracle} hold.  Fix $m\in[M]$, let
$\rho_t:=\alpha_t/\beta_{t+1}$, and suppose that
$0<\alpha_t\leq1/(2L)$, $0<\beta_{t+1}\leq1$, and $\{\rho_t\}$ is
nonincreasing with $\rho_0\leq1/(\sqrt{24M}L)$.  With
$A_T:=\sum_{t=0}^{T-1}\alpha_t$, let $x_\tau$ be the algorithm's output,
where $\tau\in\{0,\ldots,T-1\}$ is independent of the sampled batches and
$\Pr(\tau=t)=\alpha_t/A_T$. Then we have
\begin{align}
 \min_{0\leq t\leq T-1}\E\big[R_S(x_t)\big]
 &\leq\E\big[R_S(x_\tau)\big]
 =\frac1{A_T}\sum_{t=0}^{T-1}
 \alpha_t\E\big[R_S(x_t)\big] \notag\\
 &\leq\frac{20}{A_T}\left[
 f_{S,m}(x_0)-f_{S,m}^{\inf}
 +\rho_0e_0
 +\bar\sigma^2\sum_{t=0}^{T-1}\alpha_t\beta_{t+1}\right].
 \label{eq:exact-moco-ps-bound}
\end{align}
Since $e_0\leq\bar\sigma^2$, choosing
$\alpha_t=c_\alpha T^{-1/2}$ and
$\beta_{t+1}=c_\beta T^{-1/2}$ gives $O(T^{-1/2})$, whereas choosing
$\alpha_t=c_\alpha(t+1)^{-1/2}$ and
$\beta_{t+1}=c_\beta(t+1)^{-1/2}$ gives
$\widetilde O(T^{-1/2})$, provided
$0<c_\alpha\leq1/(2L)$, $0<c_\beta\leq1$, and
$c_\alpha/c_\beta\leq1/(\sqrt{24M}L)$.
\end{theorem}

The rate without a logarithmic factor uses parameters that are constant
across iterations and depend on the horizon $T$. The theorem also permits
diminishing parameter sequences.

\begin{proof}
Fix $m\in[M]$ and let $\mathcal H_t:=\sigma(Z_0,\ldots,Z_t)$ be the
history before Algorithm~\ref{alg:exact-moco} samples $Z_{t+1}$.
In particular, $x_t,Y_t,d_t$, and $x_{t+1}$ are $\mathcal H_t$-measurable.
For any finite horizon, the finite dataset and deterministic parameter
sequences imply that the iterates and gradient estimators have finite
support. Hence, all expectations below are finite. Set
$\zeta_{t+1}:=\nabla F_{Z_{t+1}}(x_{t+1})-\nabla F_S(x_{t+1})$ and
$\Delta_t:=\nabla F_S(x_t)-\nabla F_S(x_{t+1})$.
By \eqref{eq:momentum-uniform-oracle},
$\E[\zeta_{t+1}\mid\mathcal H_t]=0$ and
$\E[\|\zeta_{t+1}\|_F^2\mid\mathcal H_t]\leq\bar\sigma^2$.

Recall from \eqref{eq:momentum-tracker-notation} that 
$e_t=\E[\|E_t\|_F^2] = \E[\|Y_t-\nabla F_S(x_t)\|_F^2]$.
After subtracting $\nabla F_S(x_{t+1})$ from
\eqref{eq:exact-moco-tracker}, we take conditional second moments given
$\mathcal H_t$ and apply Young's inequality and the law of total
expectation to obtain that
\begin{align}
e_{t+1}
&=(1-\beta_{t+1})^2
\E\big[\|E_t+\Delta_t\|_F^2\big]
+\beta_{t+1}^2\E\big[\|\zeta_{t+1}\|_F^2\big]
&&\text{by conditional unbiasedness} \notag\\
&\leq(1-\beta_{t+1})e_t
+\frac{3}{\beta_{t+1}}\E\big[\|\Delta_t\|_F^2\big]
+\beta_{t+1}^2\bar\sigma^2
&&\text{by Young's inequality} \notag\\
&\leq(1-\beta_{t+1})e_t
+\frac{3ML^2\alpha_t^2}{\beta_{t+1}}
\E\big[\|d_t\|^2\big]
+\beta_{t+1}^2\bar\sigma^2
&&\text{by \eqref{eq:momentum-matrix-smoothness}
and \eqref{eq:momentum-exact-update}}.
\label{eq:exact-moco-tracker-recursion}
\end{align}
For the first inequality in \eqref{eq:exact-moco-tracker-recursion},
Young's inequality with parameter $\beta_{t+1}/2$ gives
$\|E_t+\Delta_t\|_F^2\leq(1+\beta_{t+1}/2)\|E_t\|_F^2
+(1+2/\beta_{t+1})\|\Delta_t\|_F^2$.
Multiplying by $(1-\beta_{t+1})^2$ and using $0<\beta_{t+1}\leq1$ yields
the coefficient bounds
$(1-\beta_{t+1})^2(1+\beta_{t+1}/2)\leq1-\beta_{t+1}$ and
$(1-\beta_{t+1})^2(1+2/\beta_{t+1})\leq 1 \leq 3/\beta_{t+1}$.
The noise term is bounded by
$\E[\|\zeta_{t+1}\|_F^2]\leq\bar\sigma^2$ using
\eqref{eq:momentum-uniform-oracle} and total expectation.
For the last inequality, \eqref{eq:momentum-matrix-smoothness} and
\eqref{eq:momentum-exact-update} imply
$\|\Delta_t\|_F^2\leq ML^2\|x_{t+1}-x_t\|^2
=ML^2\alpha_t^2\|d_t\|^2$.

Recall that $\rho_t=\alpha_t/\beta_{t+1}$.
Set $\rho_{-1}:=\rho_0$ and define the Lyapunov energy
$\mathcal V_t:=\E[f_{S,m}(x_t)]+\rho_{t-1}e_t$.
Taking expectations in the descent bound \eqref{eq:momentum-one-step-descent} of
Corollary~\ref{lem:momentum-descent-transfer} and applying the estimator
recursion and parameter conditions in Theorem~\ref{thm:exact-moco} gives
\begin{align}
&\mathcal V_{t+1}-\mathcal V_t \notag\\
&\leq-\frac{\alpha_t}{4}\E\big[\|d_t\|^2\big]
  +\frac{\alpha_t}{2}e_t+\rho_t e_{t+1}-\rho_{t-1}e_t
&&\text{by Corollary~\ref{lem:momentum-descent-transfer}} \notag\\
&\leq -\alpha_t\left(\frac14-\frac{3ML^2\rho_t\alpha_t}{\beta_{t+1}}\right)
  \E\big[\|d_t\|^2\big]
  +\left(\frac{\alpha_t}{2}+\rho_t(1-\beta_{t+1})-\rho_{t-1}\right)e_t
  +\rho_t\beta_{t+1}^2\bar\sigma^2
&&\text{by \eqref{eq:exact-moco-tracker-recursion}} \notag\\
&= -\alpha_t\left(\frac14-3ML^2\rho_t^2\right)
  \E\big[\|d_t\|^2\big]
  -\left(\frac{\alpha_t}{2}+\rho_{t-1}-\rho_t\right)e_t
  +\alpha_t\beta_{t+1}\bar\sigma^2
&&\text{since $\rho_t\beta_{t+1}=\alpha_t$} \notag\\
&\leq -\frac{\alpha_t}{8}\E\big[\|d_t\|^2\big]
  -\frac{\alpha_t}{2}e_t+\alpha_t\beta_{t+1}\bar\sigma^2
&&\text{see below}.
\label{eq:exact-moco-lyapunov-descent}
\end{align}
The last inequality uses the parameter conditions in the statement of
Theorem~\ref{thm:exact-moco}: $\{\rho_t\}$ is nonincreasing and
$\rho_0\leq1/(\sqrt{24M}L)$. Thus, $\rho_t\leq\rho_{t-1}$ and
$3ML^2\rho_t^2\leq3ML^2\rho_0^2\leq1/8$, which imply
$\frac14-3ML^2\rho_t^2\geq\frac18$ and
$\frac{\alpha_t}{2}+\rho_{t-1}-\rho_t\geq\frac{\alpha_t}{2}$.
Substituting these coefficient bounds and using
$\E[\|d_t\|^2]\geq0$ and $e_t\geq0$ gives the last line of
\eqref{eq:exact-moco-lyapunov-descent}. The convention
$\rho_{-1}=\rho_0$ ensures that the same argument applies at $t=0$.
Summing \eqref{eq:exact-moco-lyapunov-descent} and using
$f_{S,m}(x_T)\geq f_{S,m}^{\inf}$ and $e_T\geq0$ give
\begin{align}
 \sum_{t=0}^{T-1}\alpha_t
 \left(\frac18\E\big[\|d_t\|^2\big]+\frac12e_t\right)
 &\leq f_{S,m}(x_0)-f_{S,m}^{\inf}+\rho_0e_0
 +\bar\sigma^2\sum_{t=0}^{T-1}\alpha_t\beta_{t+1}
 &&\text{by telescoping}
 \label{eq:exact-moco-joint-bound}
\\
 \frac1{A_T}\sum_{t=0}^{T-1}
 \alpha_t\E\big[R_S(x_t)\big]
 &\leq
 \frac2{A_T}\sum_{t=0}^{T-1}\alpha_t
 \left(\E\big[\|d_t\|^2\big]+e_t\right)
 &&\text{by Corollary~\ref{lem:momentum-descent-transfer},
 \eqref{eq:momentum-residual-transfer}} \notag\\
 &\leq\frac{20}{A_T}\left[
 f_{S,m}(x_0)-f_{S,m}^{\inf}+\rho_0e_0
 +\bar\sigma^2\sum_{t=0}^{T-1}\alpha_t\beta_{t+1}\right]
 &&\text{by \eqref{eq:exact-moco-joint-bound}}.
 \label{eq:exact-moco-residual-sum}
\end{align}
By the sampling distribution of $\tau$, the weighted average equals
$\E[R_S(x_\tau)]$, proving \eqref{eq:exact-moco-ps-bound}.
Finally, \eqref{eq:momentum-uniform-oracle} gives
$e_0\leq\bar\sigma^2$.  For the horizon-dependent schedule,
$A_T=c_\alpha T^{1/2}$ and
$\sum_{t=0}^{T-1}\alpha_t\beta_{t+1}=c_\alpha c_\beta$; for the
diminishing schedule, $A_T=\Theta(T^{1/2})$ and the latter sum is
$O(1+\log T)$.  These relations establish the stated rates.
\end{proof}

\subsection{MoCo+ with an exact MGDA update}

MoCo+ uses a corrected momentum gradient estimator of STORM type with
coefficient $\gamma_{t+1}\in(0,1]$. At iteration $t$, the exact MGDA update
\eqref{eq:momentum-exact-weight}--\eqref{eq:momentum-exact-update} first
determines $x_{t+1}$. A new batch $Z_{t+1}$ is then sampled, and the
gradient matrices at $x_t$ and $x_{t+1}$ are evaluated using this batch:
\begin{equation}
 Y_{t+1}
 :=\nabla F_{Z_{t+1}}(x_{t+1})
 +(1-\gamma_{t+1})
 \bigl(Y_t-\nabla F_{Z_{t+1}}(x_t)\bigr),
 \qquad 0<\gamma_{t+1}\leq1.
 \label{eq:exact-mocop-tracker}
\end{equation}
Using the same batch at both points allows the variance of the gradient
difference to be bounded in terms of $\|x_{t+1}-x_t\|^2$.
The proof first establishes the resulting estimator error recursion,
then controls the change in the inverse stepsize. These bounds are
combined with Corollary~\ref{lem:momentum-descent-transfer} through a
Lyapunov function that weights the mean squared estimation error by the
inverse stepsize. Algorithm~\ref{alg:exact-mocop} specifies the sampling
order and the exact MGDA update.

\begin{algorithm}[H]
\small
\caption{MoCo+ with an exact MGDA update}
\label{alg:exact-mocop}
\begin{algorithmic}[1]
\State \textbf{Input:} training dataset $S$, initial point $x_0$, horizon
$T$, batch size $|Z_t|$, stepsizes $\{\alpha_t\}_{t=0}^{T-1}$, and momentum
coefficients $\{\gamma_{t+1}\}_{t=0}^{T-1}$.
\State Sample $Z_0$ and set $Y_0=\nabla F_{Z_0}(x_0)$.
\For{$t=0,\ldots,T-1$}
  \State Compute $\lambda_t$ and $d_t$ by
  \eqref{eq:momentum-exact-weight}--\eqref{eq:momentum-exact-direction}.
  \State Update $x_{t+1}$ by \eqref{eq:momentum-exact-update}.
  \State Sample a new batch $Z_{t+1}$ and evaluate its gradient matrices at
  both $x_t$ and $x_{t+1}$.
  \State Update $Y_{t+1}$ by \eqref{eq:exact-mocop-tracker}.
\EndFor
\State Sample $\tau\in\{0,\ldots,T-1\}$ independently with
$\Pr(\tau=t)=\alpha_t/\sum_{s=0}^{T-1}\alpha_s$.
\State \textbf{Output:} $x_\tau$.
\end{algorithmic}
\end{algorithm}

\begin{theorem}[Convergence of MoCo+ with an exact MGDA update]
\label{thm:exact-mocop}
Let $\{x_t,Y_t\}_{t=0}^T$ be generated by Algorithm~\ref{alg:exact-mocop}
from a deterministic $x_0$, with $Y_0=\nabla F_{Z_0}(x_0)$, fixed batch size
$|Z_t|\equiv|Z_0|\geq1$, and horizon $T\geq1$.
Suppose that Assumptions~\ref{ass:smooth-lower}--\ref{ass:oracle} and
\eqref{eq:momentum-uniform-oracle} hold.  Fix $m\in[M]$ and
$\nu\in[1/3,1/2]$ (the stepsize decay exponent), and set
$\alpha_t=\alpha_0(t+1)^{-\nu}$,
$\gamma_{t+1}=c_\gamma\alpha_t^2$,
$c_\nu=\nu 2^\nu/\alpha_0^2$, and
$\kappa=(c_\gamma-c_\nu)^{-1}$, where
$0<\alpha_0\leq\sqrt{1-\nu2^\nu}/(4\sqrt M\,L)$ and
$c_\nu+16ML^2\leq c_\gamma\leq1/\alpha_0^2$.
With $A_T:=\sum_{t=0}^{T-1}\alpha_t$, let $x_\tau$ be the algorithm's output,
where $\tau\in\{0,\ldots,T-1\}$ is independent of the sampled batches and
$\Pr(\tau=t)=\alpha_t/A_T$. Then
\begin{align}
 \min_{0\leq t\leq T-1}\E\big[R_S(x_t)\big]
 &\leq\E\big[R_S(x_\tau)\big]
 =\frac1{A_T}\sum_{t=0}^{T-1}
 \alpha_t\E\big[R_S(x_t)\big] \nonumber \\
 &\leq\frac{20}{A_T}\left[
 f_{S,m}(x_0)-f_{S,m}^{\inf}
 +\frac{\kappa e_0}{\alpha_0}
 +2\kappa c_\gamma^2\bar\sigma^2
 \sum_{t=0}^{T-1}\alpha_t^3\right].
 \label{eq:exact-mocop-ps-bound}
\end{align}
Consequently, for fixed $|Z_t|$, the rate is
$\widetilde O(T^{-2/3})$ when $\nu=1/3$ and
$O(T^{-(1-\nu)})$ when $1/3<\nu\leq1/2$.
\end{theorem}

Note that for $\nu\in[1/3,1/2]$, $\nu2^\nu\leq1/\sqrt2<1$. The bound on $\alpha_0$ ensures $c_\nu+16ML^2\leq1/\alpha_0^2$, so the interval for $c_\gamma$ is nonempty; it also implies $\alpha_0\leq1/(2L)$, as required by Corollary~\ref{lem:momentum-descent-transfer}.

\begin{proof}
Fix $m\in[M]$ and set $\alpha_{-1}:=\alpha_0$.
Let $\mathcal H_t:=\sigma(Z_0,\ldots,Z_t)$ be the history before
Algorithm~\ref{alg:exact-mocop} samples $Z_{t+1}$. The quantities
$x_t,Y_t,d_t$, and $x_{t+1}$ are $\mathcal H_t$-measurable, and both
gradient evaluations use the same new batch. All expectations below are
finite by the finite-support argument in the proof of
Theorem~\ref{thm:exact-moco}.
Recall from \eqref{eq:momentum-tracker-notation} that
$E_t=Y_t-\nabla F_S(x_t)$ and $e_t=\E[\|E_t\|_F^2]$.
Define $\zeta_{t+1}(x):=\nabla F_{Z_{t+1}}(x)-\nabla F_S(x)$.
By \eqref{eq:momentum-uniform-oracle}, this noise has conditional mean
zero and conditional second moment at most $\bar\sigma^2$ at both
$x=x_t$ and $x=x_{t+1}$.

Subtracting $\nabla F_S(x_{t+1})$ from
\eqref{eq:exact-mocop-tracker} yields
\begin{equation}
 E_{t+1}
 =(1-\gamma_{t+1})E_t
 +\gamma_{t+1}\zeta_{t+1}(x_t)
 +\zeta_{t+1}(x_{t+1})-\zeta_{t+1}(x_t).
 \label{eq:exact-mocop-error-identity}
\end{equation}
The noise difference is the centered stochastic gradient difference
conditional on $\mathcal H_t$. Consequently,
\begin{align}
&\E\big[\|\zeta_{t+1}(x_{t+1})-\zeta_{t+1}(x_t)\|_F^2
      \mid\mathcal H_t\big] \notag\\
&\leq  \E\big[\|\nabla F_{Z_{t+1}}(x_{t+1})
            -\nabla F_{Z_{t+1}}(x_t)\|_F^2\mid\mathcal H_t\big]
&&\text{by conditional variance} \notag\\
&\leq  ML^2\alpha_t^2\|d_t\|^2
&&\text{by \eqref{eq:momentum-matrix-smoothness} and \eqref{eq:momentum-exact-update}}.
\label{eq:exact-mocop-difference-bound}
\end{align}
In \eqref{eq:exact-mocop-error-identity}, the sum of the noise terms has
conditional mean zero given $\mathcal H_t$. Taking conditional second
moments and then total expectations gives
\begin{align}
e_{t+1} &= (1-\gamma_{t+1})^2e_t
+\E\big[\|\gamma_{t+1}\zeta_{t+1}(x_t)
        +\zeta_{t+1}(x_{t+1})-\zeta_{t+1}(x_t)\|_F^2\big]
&&\text{by \eqref{eq:exact-mocop-error-identity}} \notag\\
&\leq  (1-\gamma_{t+1})^2 e_t
+2\gamma_{t+1}^2\bar\sigma^2
+2ML^2\alpha_t^2\E\big[\|d_t\|^2\big]
&&\text{by \eqref{eq:momentum-uniform-oracle} and \eqref{eq:exact-mocop-difference-bound}}.
\label{eq:exact-mocop-tracker-recursion}
\end{align}
The last inequality also uses $\|A+B\|_F^2\leq2\|A\|_F^2+2\|B\|_F^2$
for matrices of the same dimensions.

Recall that $\alpha_t=\alpha_0(t+1)^{-\nu}$ and
$c_\nu=\nu 2^\nu/\alpha_0^2$. For $t\geq1$, we have
\begin{align}
\frac1{\alpha_t}-\frac1{\alpha_{t-1}}
&=\frac{(t+1)^\nu-t^\nu}{\alpha_0} \notag\\
&\leq  \frac{\nu}{\alpha_0}t^{\nu-1}
&&\text{by concavity of $u\mapsto u^\nu$} \notag\\
&\leq  \frac{\nu 2^\nu}{\alpha_0}(t+1)^{-\nu}
=c_\nu\alpha_t
&&\text{since $t+1\leq2t$ ~~and~~ $\nu\leq 1/2$}.
\label{eq:exact-mocop-step-ratio}
\end{align}
Indeed, $t^{\nu-1}(t+1)^\nu\leq2^{\nu}t^{2\nu-1}\leq2^\nu$.
The same bound holds at $t=0$ under the convention $\alpha_{-1}=\alpha_0$.

Recall that $\gamma_{t+1}=c_\gamma\alpha_t^2$ and
$\kappa=(c_\gamma-c_\nu)^{-1}>0$. Define the Lyapunov energy
$\mathcal W_t:=\E[f_{S,m}(x_t)]+\kappa e_t/\alpha_{t-1}$.
The inverse stepsize weight makes the estimator contraction proportional
to $\alpha_t$, since $\gamma_{t+1}/\alpha_t=c_\gamma\alpha_t$.
Taking expectations in \eqref{eq:momentum-one-step-descent} of
Corollary~\ref{lem:momentum-descent-transfer}, we obtain
\begin{align}
&\mathcal W_{t+1}-\mathcal W_t \notag\\
&\leq-\frac{\alpha_t}{4}\E\big[\|d_t\|^2\big]
  +\frac{\alpha_t}{2}e_t
  +\frac{\kappa e_{t+1}}{\alpha_t}
  -\frac{\kappa e_t}{\alpha_{t-1}}
&&\text{by Corollary~\ref{lem:momentum-descent-transfer}} \notag\\
&\leq  -\alpha_t\left(\frac14-2\kappa ML^2\right)\E\big[\|d_t\|^2\big]
  +\left[\frac{\alpha_t}{2}
    +\kappa\left(\frac{(1-\gamma_{t+1})^2}{\alpha_t}
                -\frac1{\alpha_{t-1}}\right)\right]e_t
  +\frac{2\kappa\gamma_{t+1}^2}{\alpha_t}\bar\sigma^2
&&\text{by \eqref{eq:exact-mocop-tracker-recursion}} \notag\\
&\leq  -\alpha_t\left(\frac14-2\kappa ML^2\right)\E\big[\|d_t\|^2\big]
  +\left[\frac{\alpha_t}{2}
    +\kappa\left(\frac1{\alpha_t}-\frac1{\alpha_{t-1}}\right)
    -\kappa c_\gamma\alpha_t\right]e_t
  +2\kappa c_\gamma^2\alpha_t^3\bar\sigma^2
&&\text{by $\gamma_{t+1}=c_\gamma\alpha_t^2$} \notag\\
&\leq  -\alpha_t\left(\frac14-2\kappa ML^2\right)\E\big[\|d_t\|^2\big]
  -\alpha_t\left(\kappa(c_\gamma-c_\nu)-\frac12\right)e_t
  +2\kappa c_\gamma^2\alpha_t^3\bar\sigma^2
&&\text{by \eqref{eq:exact-mocop-step-ratio}} \notag\\
&\leq  -\frac{\alpha_t}{8}\E\big[\|d_t\|^2\big]
  -\frac{\alpha_t}{2}e_t+2\kappa c_\gamma^2\alpha_t^3\bar\sigma^2
&&\text{see below}.
\label{eq:exact-mocop-lyapunov-descent}
\end{align}
The third inequality also uses $(1-\gamma_{t+1})^2\leq1-\gamma_{t+1}$.
This holds because the parameter conditions in the statement of
Theorem~\ref{thm:exact-mocop} give
$0<\gamma_{t+1}=c_\gamma\alpha_t^2\leq c_\gamma\alpha_0^2\leq1$.
For the last inequality, $c_\gamma-c_\nu\geq16ML^2$ implies
$\kappa ML^2=ML^2/(c_\gamma-c_\nu)\leq1/16$, while
$\kappa(c_\gamma-c_\nu)=1$ by definition. Hence,
$\frac14-2\kappa ML^2\geq\frac18$ and
$\kappa(c_\gamma-c_\nu)-\frac12=\frac12$.

Summing \eqref{eq:exact-mocop-lyapunov-descent} over $t=0,\ldots,T-1$,
using $\mathcal W_T\geq f_{S,m}^{\inf}$ and
$\mathcal W_0=f_{S,m}(x_0)+\kappa e_0/\alpha_0$, and applying the
stationarity bound \eqref{eq:momentum-residual-transfer} gives
\begin{align}
 \sum_{t=0}^{T-1}\alpha_t
 \left(\frac18\E\big[\|d_t\|^2\big]+\frac12e_t\right)
 &\leq f_{S,m}(x_0)-f_{S,m}^{\inf}
 +\frac{\kappa e_0}{\alpha_0}
 +2\kappa c_\gamma^2\bar\sigma^2\sum_{t=0}^{T-1}\alpha_t^3
 &&\text{by telescoping}
 \label{eq:exact-mocop-joint-bound}\\
 \frac1{A_T}\sum_{t=0}^{T-1}\alpha_t\E\big[R_S(x_t)\big]
 &\leq\frac2{A_T}\sum_{t=0}^{T-1}\alpha_t
 \left(\E\big[\|d_t\|^2\big]+e_t\right)
 &&\text{by Corollary~\ref{lem:momentum-descent-transfer}} \notag\\
 &\leq\frac{20}{A_T}\left[
 f_{S,m}(x_0)-f_{S,m}^{\inf}+\frac{\kappa e_0}{\alpha_0}
 +2\kappa c_\gamma^2\bar\sigma^2\sum_{t=0}^{T-1}\alpha_t^3\right]
 &&\text{by \eqref{eq:exact-mocop-joint-bound}}.
 \label{eq:exact-mocop-residual-sum}
\end{align}
By the sampling distribution of $\tau$, the weighted average equals
$\E[R_S(x_\tau)]$, proving \eqref{eq:exact-mocop-ps-bound}.
Finally, $e_0\leq\bar\sigma^2$ by \eqref{eq:momentum-uniform-oracle}.
For $\nu=1/3$, $A_T=\Theta(T^{2/3})$ and
$\sum_{t=0}^{T-1}\alpha_t^3=O(1+\log T)$, which yields
$\widetilde O(T^{-2/3})$. For $1/3<\nu\leq1/2$,
$A_T=\Theta(T^{1-\nu})$ and $\sum_{t=0}^{\infty}\alpha_t^3<\infty$,
which yields $O(T^{-(1-\nu)})$. In both cases, the rates bound the expected
squared empirical PS measure at the algorithm's output.
\end{proof}

For fixed $|Z_t|$, Algorithm~\ref{alg:exact-moco} requires
$|Z_t|(T+1)$ stochastic gradient matrix evaluations.
Algorithm~\ref{alg:exact-mocop} requires $|Z_t|(2T+1)$ such evaluations
because each corrected update evaluates gradients at two points using the
same batch. Both evaluation counts are $\Theta(T)$. Consequently, the
convergence exponents in Theorems~\ref{thm:exact-moco}--\ref{thm:exact-mocop}
are unchanged when the rates are expressed in terms of the number of
stochastic gradient evaluations. For $M=1$, the MoCo estimator is an
exponentially weighted stochastic gradient, and the MoCo+ estimator follows
the STORM recursion.

\section{Proof discovery process}
\label{app:proof-discovery}

The proof in the main text was discovered while the author was designing homework for a graduate course on multi-objective machine learning.
Problem~1, ``MGDA,'' asked students to derive the MGDA update equation and the associated properties that could be useful in solving Problem~2.
Problem~2, ``Stochastic MGDA,'' asked students to analyze stochastic MGDA under smoothness, conditional unbiasedness, bounded conditional variance, and an increasing mini-batch schedule. The box below records the author-created problem statements supplied to ChatGPT together with the prompt.

The initial expectation was that the model would recover an argument from closely related papers~\citep{svaiter2018holder,chen2024threeway}. It instead returned a different route based on the stated problem and surrounding context.
In response to the prompt, ChatGPT~5.4 TE proposed the key route of combining stochastic common descent with the Lipschitz continuity of the PS measure so that the mini-batch error enters through its second moment.
The author subsequently verified, corrected, and reorganized the proof.

The homework excerpt below uses the notation of
Section~\ref{sec:problem-setting}: $p$ is the parameter dimension,
$M$ is the number of objectives, $[M]=\{1,\ldots,M\}$, and $S$ is the
training dataset. For objective $m$, $f_{z,m}$ is the loss on datum $z$,
$f_{S,m}$ is its empirical average, and $\nabla F_S$ collects the empirical
gradients as columns. The symbols $x_t$ and $\alpha_t>0$ denote the iterate
and stepsize, $L$ is the gradient Lipschitz constant, and $\mathbf 1$ is the
all-ones vector. Empirical Pareto stationarity means $R_S(x)=0$, with $R_S$
defined in \eqref{eq:ps-measure}.

\begin{promptbox}
\textbf{Homework input supplied to ChatGPT (problem statements only).} The following statements were supplied to ChatGPT as an attached homework LaTeX file together with the prompt. Their notation is aligned with this manuscript.

\smallskip
\subsection*{Problem 1: MGDA}
For the multi-objective optimization problem $\min_{x\in\R^p} F_S(x):=[f_{S,1}(x),\ldots,f_{S,M}(x)]^\top$ with differentiable $F_S$,
the multi-gradient descent algorithm (MGDA) update is
\begin{align}
&d_{\nabla F_S(x)}:=\mathop{\arg\min}_{d\in\R^p}\;\max_{m\in[M]}
\left\{\langle \nabla f_{S,m}(x),d\rangle+\tfrac{1}{2}\|d\|^2\right\}, \notag\\
&x_{t+1}=x_t+\alpha_t d_{\nabla F_S(x_t)}.
\end{align}
Prove the following: \\
(1) (10 pts) the update direction $d_{\nabla F_S(x)}$ can be computed as a weighted combination of the gradients, i.e., $d_{\nabla F_S(x)}=-\nabla F_S(x)\lambda^\star_{\nabla F_S(x)}$,
where $\lambda^\star_{\nabla F_S(x)}\in\Delta^M:=\{\lambda\in\R^M\mid\lambda\geq 0,\ \mathbf{1}^\top\lambda=1\}$;
state how $\lambda^\star_{\nabla F_S(x)}$ can be obtained from $\nabla F_S(x)$; \\
(2) (20 pts) if $f_{S,m}(x)$ is $L$-smooth for all $m\in[M]$, the direction $d_{\nabla F_S(x)}$ is H{\"o}lder continuous with respect to $x$, i.e.,
$\|d_{\nabla F_S(x_1)}-d_{\nabla F_S(x_2)}\|\leq c\|x_1-x_2\|^{\frac{1}{2}}$ for some $c>0$; \\
(3) (15 pts) if $f_{S,m}(x)$ is $L$-smooth for all $m\in[M]$, the norm of the direction $\|d_{\nabla F_S(x)}\|$ is Lipschitz continuous with respect to $x$, i.e.,
$\big|\|d_{\nabla F_S(x_1)}\|-\|d_{\nabla F_S(x_2)}\|\big|\leq c\|x_1-x_2\|$ for some $c>0$. \\
\subsection*{Problem 2: Stochastic MGDA}
(30 pts) For the stochastic variant of MGDA,
we obtain the mini-batch gradient matrix $Q_t:=\nabla F_{Z_t}(x_t)$ with $Z_t$ denoting the mini-batch data at each iteration $t$.
Assume $f_{z,m}(x)$ is $L$-smooth for every $z\in S$ and $m\in[M]$,
and the mini-batch size $|Z_t|$ increases with $t$,
i.e., $|Z_t|= \Theta(t+1)$,
prove the convergence of this stochastic variant of MGDA to empirical Pareto stationarity for $F_S$, and state the convergence rate.
\smallskip
\vspace{5mm}
\textbf{Prompt to ChatGPT~5.4 TE.}
\begin{quote}
\small\ttfamily
create homework3 solutions and modify this tex file with completed solutions
\end{quote}
\end{promptbox}

\section{Student homework submissions summary}
\label{app:student-submissions}

The preceding section explains how the proof in this note emerged from the homework prompt. This section compares the six sampled graded submissions, focusing on their proof routes and whether the claimed convergence rates are
correctly justified.

\subsection{Proof routes and correctness}

Table~\ref{tab:student-homework-summary} summarizes the six graded submissions for Problems~1 and~2, with details in Appendix~\ref{app:proof-discovery}. The submissions are anonymized as S1--S6 to protect student privacy. Five records document an LLM version and at least one prompt. All six submissions correctly identified the minimum-norm simplex quadratic program in Problem~1(1), and most correctly proved the other properties in Problem~1, while the main differences arose in the proofs for Problem~2, concerning the convergence of stochastic MGDA.

\begingroup
\uniformtableformat
\scriptsize
\setlength{\tabcolsep}{4pt}
\setlength{\LTcapwidth}{\textwidth}
\begin{longtable}{m{0.05\textwidth}m{0.1\textwidth}
  >{\centering\arraybackslash}m{0.055\textwidth}
  >{\centering\arraybackslash}m{0.14\textwidth}m{0.545\textwidth}}
\caption{Anonymous records for the six homework submissions. A \cmark\ marks the proof route used in this note. In the rate column, ``\emph{Claimed}'' means the derivation does not justify the stated rate, ``\emph{Proved}'' means it does under $|Z_t|=\Theta(t+1)$, ``\emph{Incomplete}'' means a rate calculation is started but not completed, and ``\emph{No rate}'' means no final rate is stated.}
\label{tab:student-homework-summary}\\[6pt]
\toprule
Record & LLM disclosure & \makecell[c]{Same\\route?} & \makecell[c]{Final rate for\\$|Z_t|=\Theta(t+1)$} & Assessed stochastic proof route \\
\midrule
\endfirsthead
\toprule
Record & LLM disclosure & \makecell[c]{Same\\route?} & \makecell[c]{Final rate for\\$|Z_t|=\Theta(t+1)$} & Assessed stochastic proof route \\
\midrule
\endhead
\endfoot
\bottomrule
\endlastfoot
S1 & ChatGPT~5.4 Pro & \xmark & \makecell[c]{Claimed:\\$O(T^{-1/4})$} & Uses the $(1/2)$-H\"older continuity of $d_Q$ before telescoping, but the assumptions are incomplete. \\
\midrule
S2 & \xmark & \xmark & \makecell[c]{Claimed:\\$O(T^{-1/2})$} & Assumes a Lipschitz MGDA direction without the required conditioning assumptions. \\
\midrule
S3 & ChatGPT~5.4 TE & \cmark & Incomplete & Uses the same route as this note, but the rate calculation is incomplete. \\
\midrule
S4 & Gemini Pro & \cmark & \makecell[c]{Proved:\\$\widetilde O(T^{-1})$} & Uses the same route as this note, and correctly proves the result. \\
\midrule
S5 & ChatGPT~5.2 Plus & \xmark & No rate & Uses $\alpha_t=\alpha_0/(t+1)$ and telescopes a descent inequality for the stochastic direction to obtain asymptotic convergence of the minimum PS measure over the iterates. It gives no finite-time rate and assumes an accumulation point without proving boundedness. \\
\midrule
S6 & ChatGPT~5 & \xmark & \makecell[c]{Claimed:\\$\widetilde O(T^{-1/2})$} & Applies deterministic common descent to the stochastic update, adds an $O(\alpha_t\E\|E_t\|_F)$ error term, and telescopes with $\alpha_t=t^{-1/2}$ to claim the rate. The response attempts to prove Lipschitz continuity of $x\mapsto r(\nabla F(x))$, rather than the required map $Q\mapsto r(Q)$. The argument is incomplete. \\
\end{longtable}
\endgroup

\begin{samepage}
\subsection{Main observations}

The comparison highlights three points that connect the submissions to the present theorem.

\begin{itemize}
\setlength{\itemsep}{2pt}
\setlength{\parskip}{0pt}
\setlength{\parsep}{0pt}
\item \textbf{Successful route.} S3 and S4, which used ChatGPT~5.4 TE
and Gemini Pro, respectively, obtained the same proof route as this note.
S3's final convergence-rate calculation was incomplete,
whereas S4 obtained the same rate as this note.
\item \textbf{Useful context.} Problem~1(3) already asked for Lipschitz continuity of $\|d_{\nabla F_S(x)}\|$ with respect to $x$.
Although this slightly differs from the 
continuity of $\|d_Q\|$ with respect to $Q$
in Lemma~\ref{lem:ps-measure-lipschitz},
it may have provided relevant context to the LLMs.
In this sample, two LLM-assisted submissions assembled the same proof route 
as this note.
\item \textbf{``Better'' model.} 
ChatGPT~5.4 TE and Gemini Pro performed the best among the sampled homework submissions.
This suggests models with extended and deeper ``thinking ability'' may perform better. However, the sample size is insufficient for a conclusion.
\end{itemize}
\end{samepage}

\subsection{Recurring mistakes in LLM-assisted derivations}

The incomplete or incorrect arguments fail at the following specific steps, and the record labels and model disclosures match Table~\ref{tab:student-homework-summary}.

\begin{itemize}
\setlength{\itemsep}{2pt}
\setlength{\parskip}{0pt}
\setlength{\parsep}{0pt}
\item \textbf{Stochastic descent (S6: ChatGPT~5).} The solution applies deterministic common descent to the stochastic update without deriving the resulting error term. Unbiased gradients do not imply $\E[d_{Q_t}\mid x_t]=d_{\nabla F_S(x_t)}$ because $d_Q$ depends nonlinearly on $Q$ and because the mini-batch simplex minimizer and $E_t$ depend on the same sample.
\item \textbf{Regularity (S2: no LLM disclosed and S6: ChatGPT~5).} S2 assumes a Lipschitz MGDA direction without the required conditioning assumptions, although it is generally only $(1/2)$-H\"older continuous. S6's argument uses $x\mapsto r(\nabla F(x))$ and does not prove the required Lipschitz continuity of $Q\mapsto r(Q)$.
\item \textbf{Assumptions and rates (S1: ChatGPT~5.4 Pro and S3: ChatGPT~5.4 TE).} S1 uses the $(1/2)$-H\"older continuity of $d_Q$ to claim an $O(T^{-1/4})$ rate but does not state all assumptions needed for its bounds. S3 follows the proof route of this note but does not complete the rate calculation.
\item \textbf{Asymptotic conclusion (S5: ChatGPT~5.2 Plus).} The solution uses $\alpha_t=\alpha_0/(t+1)$ to obtain asymptotic convergence of the minimum PS measure over the iterates, but gives no finite-time rate and assumes an accumulation point without proving boundedness.
\end{itemize}

\newpage

\printbibliography

\end{document}